\documentclass[11pt, notitlepage]{article}
\usepackage{amsmath, amsthm, amssymb, amstext, comment}
\catcode`\@=11 \@addtoreset{equation}{section}

\def\theequation{\thesection.\arabic{equation}}
\catcode`\@=12
\usepackage{colortbl}%
\usepackage[mathscr]{eucal}
\usepackage{epsf}
\usepackage{a4wide}
\usepackage{hyperref,xcolor}
\hypersetup{
	pdfborder={0 0 0},
	colorlinks,
}
\usepackage[left=2cm,right=2cm,top=1.5cm,bottom=1cm,includeheadfoot]{geometry}

\newcommand{\ds} {\displaystyle}
\newcommand{\e}{\epsilon}

\newcommand{\al} {\alpha}

\newcommand{\de} {\delta}
\newcommand{\ga} {\gamma}
\newcommand{\Ga} {\Gamma}
\newcommand{\Om} {\Omega}

\newcommand{\De} {\Delta}
\newcommand{\la} {\lambda}
\newcommand{\La} {\Lambda}
\newcommand{\noi} {\noindent}

\newcommand{\uline} {\underline}
\newcommand{\oline} {\overline}
\newcommand{\mb} {\mathbb}
\newcommand{\mc} {\mathcal}

\usepackage[all]{xy}
\catcode`\@=11
\def\theequation{\@arabic{\c@section}.\@arabic{\c@equation}}
\catcode`\@=12

\newtheorem{theorem}{Theorem}
\newtheorem{remark}[theorem]{Remark}
\newtheorem{lemma}[theorem]{Lemma}
\newtheorem{proposition}[theorem]{Proposition}

\newtheorem{definition}[theorem]{Definition}
\newtheorem{example}[theorem]{Example}

\numberwithin{theorem}{section} \numberwithin{equation}{section}
\renewcommand{\theequation}{\arabic{section}.\arabic{equation}}

\begin{document}
	
	{\vspace{0.01in}}
	
	\title
	{ \sc Modified quasilinear  equations with strongly singular and critical exponential nonlinearity}
	
	\author{Reshmi Biswas\footnote{Departamento de Matem\'atica, Universidad T\'ecnica Federico Santa Mar\'ia, Avda. Espana, 1680-Valpara\'iso, Chile.
			e-mail: reshmi15.biswas@gmail.com}~, ~~Sarika Goyal\footnote
		{
			Department of Mathematics, Netaji Subhas University of Technology,
			Dwarka Sector-3, Dwarka, New Delhi 110078, India.
			e-mail: sarika1.iitd@gmail.com }~ and ~K. Sreenadh\footnote{Department of Mathematics, Indian Institute of Technology Delhi,
			Hauz Khaz, New Delhi 110016, India. \\
			Currently at: Indian Institute of Technology Delhi - Abu Dhabi,
			Zayed City, Abu Dhabi, UAE.
			e-mail: sreenadh@gmail.com} }

	\date{}

	\maketitle

	%
	%
	%
	%
	%
	%
	\begin{abstract}\noi In this paper, we study  global multiplicity result for a class of modified quasilinear  singular equations involving  the  critical exponential growth: \begin{equation*}
			\left\{
			\begin{array}{rlll}
				-\Delta u -\Delta (u^{2})u&=&\la \left(\al(x) u^{-q} + f(x,u)\right) \; \text{in}\;
				\Om,\\
				u&>& 0~\mbox{in}~ \Om, \\
				u&=&0 ~\mbox{on}~ \partial\Om,
			\end{array}
			\right.
		\end{equation*}
		where $\Om$ is a smooth bounded domain in $\mathbb R^2,\;$ $0<q<3$ and $\al:\Om\to (0,+\infty)$  such that $\al\in L^\infty(\Om)$.  The function $f:\Om \times \mb R\to \mb R$ is continuous and enjoys critical exponential growth of  the Trudinger–Moser type. Using a sub-super solution method, we show that there exists some $\La^*>0$ such that for all $\la\in(0,\La^*)$ the problem has at least two positive solutions, for $\la=\La^*$, the problem achieves at least one positive solution and for $\la>\La^*,$ the problem has no solutions. 
		\medskip
		
		\noi \textbf{Key words:} Modified quasilinear operator,  singular nonlinearity, sub-super solution, critical exponential nonlinearity, Trudinger-Moser inequality.

		\medskip
		
		\noi \textit{2010 Mathematics Subject Classification:} 35J20, 35J60
	\end{abstract}

	
	

		
		%
		%
		%
		\section{\bf Introduction and statement of main results}

		\noi In this article, we study  the existence, nonexistence and multiplicity of the positive solutions for the following  modified quasilinear  equation:
		\begin{equation*}
			\label{pq}\tag{$P_{*}$}\left\{
			\begin{array}{rlll}
				-\Delta u -\Delta (u^{2})u&=&\la \left(\al(x) u^{-q} + f(x,u)\right) \; \text{in}\;
				\Om,\\
				u&> &0~\mbox{in}~ \Om,\\
				u&=&0 ~\mbox{on}~ \partial\Om,
			\end{array}
			\right.
		\end{equation*}
		where $\Om$ is a smooth bounded domain in $\mathbb R^2,\;$ $0<q<3$ and the function $f:\Om \times \mb R\to \mb R$ is defined as $f(x,s)=g(x,s)\exp(|s|^{4})$, where $g \in C(\bar \Om \times \mb R)$ satisfies some appropriate assumptions described later. We also have the following assumption on the  function $\al:\Om\to \mb R$:
		\begin{itemize} 
			\item[$(\al1)$ ]  $\al\in L^{\infty}(\Om)$ and $\al_0:=\inf_{x\in \Om}\al(x)>0$.
		\end{itemize} 
		
		\noi  The study on the equations driven by the modified quasilinear  operator $-\Delta u -\Delta (u^{2})u$ is quite popular  for long because of their wide range of applications in the modeling of the physical phenomenon such as in plasma physics and fluid mechanics \cite{bass}, in dissipative quantum mechanics \cite{has}, etc. 	Solutions of such equations (called soliton solutions) are related to the existence of standing wave solutions for quasilinear Schr\"odinger equations of
		the form 
		\begin{align}\label{aa1}
			iu_t =-\Delta u+V(x)u-h_1(|u|^2)u-C\Delta h_2(|u|^2)h_2'(|u|^2)u,\;\; x\in\mathbb R^N,
		\end{align}
		where $V$ is a  potential function, $C$ is a real constant, $h_1$ and $h_2$ are real valued functions.  Equations of the form \eqref{aa1} appear in  the study of mathematical physics.
		Each different type of the function $h_2$ represents different physical phenomenon.  For example, if $h_2(s)=s$, then \eqref{aa1}  is used in the modeling of the super fluid film equation in plasma physics (see \cite{1}). When $h_2(s) =\sqrt{1+s^2}$, \eqref{aa1} attributes to the study of   self-channeling of a high-power ultra short laser in matter (see \cite{33}). Because of the quasilinear term $\De(u^2) u$, present in the problems of type \eqref{pq}, the natural energy functional associated to such problems is not well defined. Hence, we have a restriction in applying variational method directly for studying such problems. To overcome this shortcoming,  researchers developed several methods and arguments, such as the perturbation method (see for e.g.,\cite{19,24}) a constrained minimization
		technique (see for e.g., \cite{20,22,28,29}),  and a
		change of variables (see for e.g., \cite {CJ,do,8,M-do1,11,12}).\\
		
		\noi Without the quasilinear term $\Delta (u^{2})u$, the problem \eqref{pq} goes back to the original semilinear equation
		\begin{equation}\label{21}
			-\Delta u=\la \al(x) u^{-q} + f(x,u) \; \text{in}\;
			\Om,\;
			u> 0~\mbox{in}~ \Om,\;
			u=0 ~\mbox{on}~ \partial\Om.
		\end{equation}Such kind of equations have significant applications in the physical modelling related to the boundary layer phenomenon for viscous fluids, non-Newtonian fluids, etc. If $f\equiv0$,  \eqref{21} becomes purely singular problem, for which existence, uniqueness, non-existence and regularity results are extensively studied in \cite{crandal ,hirano}  with suitable  $\al(x)$ and for different ranges of $q$.\\
		One of the main features of problem \eqref{pq} is that the 
		nonlinear term $f(x,s)$ enjoys the critical exponential growth with respect
		to the following Trudinger–Moser inequality (see \cite{Trud-Moser}):
		\begin{theorem}\label{TM-ineq}(\textbf {Trudinger–Moser inequality})
			Let $\Om$ be a  smooth bounded domain in $\mb R^2.$ Then for  $u \in H^{1}_0(\Om)$ and $p>0$ we have
			\[ \exp(p|u|^{2})\in L^1(\Om).\] Moreover,
			\[\sup_{\|u\|\leq 1}\int_\Om \exp(p|u|^{2})~dx < +\infty\]
			if and only if $p \leq 4\pi$.
		\end{theorem}
		\noi  Here the Sobolev space $H^{1}_0(\Om)$ and the corresponding norm $\|\cdot\|$
		are defined in Section \ref{sec2}. The study on the critical growth exponential problems associated to this inequality were initiated with the work of  Adimurthi \cite{adi} and de Figueiredo et al. \cite{DR}. Furthermore, 
		the problem of type \eqref{pq} without the singular term, that is the equation \begin{align}\label{exp}-\Delta u -\Delta (u^{2})u + V(x) u=f(x, u) \,\; \text{in}\;\; \mb R^2,
		\end{align} where $V:\mb R^N\to \mb R$ is a continuous potential and $f:\mb R^2\times\mb R\to \mb R$ is  a continuous function with some suitable assumption and is having critical exponential growth ($\exp(p s^{4})$), was studied by  do Ó et al. in \cite{do} for the first time. Note that, unlike as in the case of semilinear critical exponential problem involving Laplacian, where the critical exponential growth is given as $\exp(|s|^{2})$, in problem \eqref{pq} and in \eqref{exp}, the form of critical exponential nonlinearity $f$ is  considered as  $\exp(|s|^{4})$ because of the quasilinear term $\De(u^2)u$.  \\
		
		
		\noi Motivated by the seminal work of Ambrosetti et al. \cite{abc}, many researchers studied the global multiplicity results for singular-convex problems. In \cite{haito}, the author proved the global multiplicity result for \eqref{21}, considering critical polynomial growth as  $f(x,s)=|s|^{\frac{N+2}{N-2}}$, and $0<q<1$ in a general smooth bounded domain $\Om\subset \mb R^N$, $N> 2.$ Then, in the critical dimension $N=2$, Adimurthi and Giacomoni \cite{adijg} discussed the global multiplicity of $H^1_0$ solutions for \eqref{21}  by taking the optimal range of $q$ as $1<q<3$ and considered $f$ to be having critical exponential growth as $f(x,s)=g(x,s)\exp(4\pi s^2)$, where $g$ is some appropriate continuous function. Later, in \cite{js}, the authors showed the similar results as in \cite{haito} for \eqref{21} while considering $f(x,s)=|s|^{\frac{N+2}{N-2}}$, $0<q<3$ and adding  a smooth perturbation with sub-critical asymptotic behavior at $+\infty$. In \cite{sweta1, sweta2}, Dhanya et al. studied the global multiplicity result for the problem of type \eqref{21} with critical growth nonlinearities combined with a discontinuous function multiplied with the singular term $u^{-q}$, $0<q<3.$\\
		%

		\noi On the other hand,  for modified quasilinear Sch\"rodinger equations involving singular nonlinearity, there are a few results in the existing literature. Authors in \cite{bal, DM,santos} discussed  existence of multiple solutions of such equations with singular nonlinearity $u^{-q}$, $0<q<1$ in combination with some polynomial type perturbation. Very recently, in \cite{yang-singular}, the authors studied the global multiplicity results for the equations of type \eqref{pq} in a smooth bounded domain $\Om\subset \mb R^N,$ $N>2$, for the first time, by assuming critical polynomial growth $f(x,s)=b(x)|s|^\frac{3N+2}{N-2}$, where $b$ is a sign changing continuous function.\\ 

		\noi Inspired by all these aforementioned works, in this article, we investigate the existence, multiplicity and non-existence results (that is, the global multiplicity result) for  problem \eqref{pq} driven by modified quasilinear operator involving  singular and critical exponential nonlinearity, which was an open question.  The main  mathematical difficulties we face in studying  problem  \eqref{pq} occur in three folds as following:
		\begin{itemize}
			\item[1.] The modified quasilinear term $\Delta(u^2)u$ prohibits the natural energy functional corresponding to 
			problem \eqref{pq} to be  well defined for all $u\in H_0^{1}(\Om)$ (defined in Section \ref{sec2}). 
			\item[2. ] The critical exponential nature of $f$ induces non-compactness of the Palais-Smale sequences.
			\item[3. ]  The growth   of the singular nonlinearity falls into  the  range $0<q<3$, which again prevents the associated energy functional to
			problem \eqref{pq} to be  well defined for all $u\in H_0^{1}(\Om)$.
		\end{itemize} First, we  transform the problem 
		\eqref{pq} by using a change of variable as in \cite{CJ,yang-singular}. Then, applying  variational and sub-super solution methods on the transformed problem, we show that there exists some $\La^*>0$ such that for the range of the parameter $0<\la<\La^*$, the transformed problem has at least two positive solutions, for $\la=\La^*$ it achieves at least one positive solution and for $\la>\La^*$, it has no solutions. In Section 3, we prove the existence  of a non trivial weak solution for the range $0<\la\leq \La^*$ and non existence of solutions for $\la>\La^*$ by constructing a suitable sub-super solution argument. We also prove the asymptotic boundary behavior of such solutions. Then in Section 4, we investigate the existence of a second solution in a cone around the first solution for $0<\la<\La^*$ by using the Ekeland variational principle and a mountain pass argument with a min-max level. 
		There we find	 the  first critical level, below which  we prove the compactness condition of the Palais-Smale sequences. This gives rise to the existence of a second positive solution. In this whole process, we  prove many technically involved  delicate estimates due to the several complexities present in the problem \eqref{pq}. 
		We would like to mention that to the best of our knowledge, there is no work in the literature addressing the question of global multiplicity of positive  solutions  involving modified quasilinear operator and singular and exponential nonlinearity. In this article, we study the global multiplicity results for the problem \eqref{pq} up to the optimal range for the singularity ($0<q<3$).\\
		
		%
		\noi We now state all  the hypotheses imposed  on the continuous function $f:\Om\times\mb R\to\mb R$, given by $f(x,s)= g(x,s)\exp\left(|s|^{4}\right)$:
		\begin{enumerate}
			\item[$(f1)$] $g\in {C^1(\overline{\Om}\times \mb R)}$ such that for each $x\in\overline\Om$,  $g(x,s)=0$ for all $s\le 0$; $g(x,s)>0$ \text{for all}
			$s>0$  and {$\frac {f(x,s)}{s^3}$ is non decreasing in $s>0$,} for all $x\in\overline\Om$. 
			\item[$(f2)$] {\it Critical growth assumption}: For any $\e>0$, $$\displaystyle\lim_{s\to +\infty}\sup_{x\in\overline\Om}g(x,s)\exp\left(-\e|s|^{4}\right)=0\text{\; \; and \;\;} \displaystyle\lim_{s\to +\infty}\inf_{x\in\overline\Om}g(x,s) {\exp(\e |s|^{4})}=+\infty.$$
			\item[$(f3)$]  There exists a constant $\tau>4$  such that $0<\tau F(x,s)\leq f(x,s)s,$ for all $ (x,s)\in \Om\times (0,+\infty).$
			\item[$(f4)$] There exists a constant $M_1>0$ such that $F(x,s)\leq M_1(1+f(x,s))$ for all $s>0$.
		\end{enumerate}
		\begin{example}
			Consider $f(x,s)= g(x,s)e^{s^{4}}$, where $g(x,s)=\left\{\begin{array}{lr}
				{t^{a_0+ 1} \exp(d_0s^{r})},\; \mbox{if} \; s> 0\\
				0, \; \mbox{if} \; s\leq 0	\end{array}\right.$
			for some $a_0>0$, $0<d_0\leq 4\pi$ and $1\leq r <4$. Then $f$ satisfies all the conditions from $(f1)$-$(f4)$.
		\end{example}
		\begin{remark}\label{rf1}
			From the condition $(f_1)$,  we deduce
			\[0\leq \lim_{s\to 0+}\frac{f(x,s)}{s}\leq \lim_{s\to 0+}\frac{f(x,s)}{s^3}s^2\leq \lim_{s\to 0+}{f(x,1)}{s^2}=0\,\,\, \text {uniformly in $x\in\Om$,}\]  since $g$ is continuous, which implies that   
			\begin{align}\label{rmk1}
				\lim_{s\to 0}\frac{f(x,s)}{s}=0\,\,\,\,\,\, \text{ uniformly in $x\in\Om$.}
			\end{align} 
			Moreover, the condition $(f1)$ yields that  $g(x,s)$ is non decreasing in $s$ and hence, $g^\prime(x,s):=g_s(x,s)\geq0$ so that 
			\begin{align*}
				f^\prime(x,s):=f_s(x,s)=(g^\prime(x,s)+4s^3g(x,s))\exp(s^4)\geq 4s^3g(x,s))\exp(s^4)=4s^3f(x,s).
			\end{align*} Thus, for any $M_0>\sqrt{2},$ there exists  $L>0$ such that\begin{align}\label{rmk2}f'(x,s)\geq M_0f(x,s)-L\,\,\,\text{  for all } s>0.\end{align}
		\end{remark}

		\noi		Now for any $\phi \in C(\overline \Om)$ with $\phi>0$ in $\Om$, we define the set
		\[C_\phi:=\{u\in C_0(\overline\Om) | \text{ there exists } c\geq 0\ \text { such that } |u(x)|\leq c\phi \text { for all } x\in \Om\}\] equipped with the norm $\|u\|_{C_\phi(\overline \Om)}:=\|u/\phi\|_\infty.$  Next, we define the following open convex set of $C_\phi(\Om)$  as 
		\[C_\phi^+(\Om)=\left\{u\in C_\phi (\Om)| \inf_{x\in\Om}\frac{u(x)}{\phi(x)}>0\right\}.\] That is, the set $C_\phi^+(\Om)$ is consisting of all those functions $u\in C(\Om)$ such that $C_1 \phi\leq u\leq C_2 \phi$ in $\Om$ for some $C_1,C_2>0$. Let us also define the distance function $$\de(x):=dist(x,\partial \Om)=\ds\inf_{y\in\partial\Om}|x-y|,\,\,\, \text{for any $x\in \overline \Om$}.$$  
		 We consider the following  eigenvalue problem :
		\begin{align}\label{ev}
			-\De u =\la u \text {  in } \Om,\;
			u>0 \text {  in } \Om,\;u=0 \text {  on } \partial \Om.
		\end{align}
		Let $\varphi_{1,\Om}\in H^1_0(\Om)$ be a positive  solution (first eigenfunction)  of the above equation corresponding to the first eigenvalue $\tilde\la_{1,\Om}$   with $\Vert\varphi_{1,\Om}\Vert_{\infty}<1$.
		We recall that $\varphi_{1,\Om}\in C^{1,\theta}(\overline \Om)$ for some $\theta\in (0,1)$ and $\varphi_{1,\Om}\in C_{\de}^+(\Omega).$ 
		{ For more properties of $\varphi_{1,\Om}$, one can refer to \cite{Drabek} }
		Then, we define  the function $\varphi_q$ as follows:
		\begin{equation*}
			\varphi_q= \left\{\begin{array}{l}\displaystyle
				\varphi_{1,\Om}\quad\mbox{ if } 0<q<1,\nonumber\\
				\varphi_{1,\Om}\left(\log\left(\frac{2}{\varphi_{1,\Om}}\right)\right)^{\frac{1}{q+1}}\quad\mbox{ if } q=1,\nonumber\\
				\varphi_{1,\Om}^{\frac{2}{q+1}}\quad\mbox{ if } q>1.
			\end{array}\right.
		\end{equation*} 
		\\\\
		Now we state the main results of this article:
		\begin{theorem}\label{thm1}
			Let $\Om\subset \mb R^2$ be a smooth bounded domain. Suppose that the hypotheses $(f1)-(f4)$ and $(\al1)$ hold. Then there exists $\La^*>0$ such that
			for every $\la\in(0,\La^*]$, problem \eqref{pq} has at least one solution in {$H_0^1(\Om)\cap C^+_{\varphi_q}(\Omega)$}
			and  for $\la>\La^*$, problem \eqref{pq} has no solutions.
		\end{theorem}
		\begin{theorem}\label{thm2}
			Let $\Om\subset \mb R^2$ be a smooth bounded domain. Assume that the hypotheses $(f1)-(f4)$ and $(\al1)$ hold. Then for every $\la\in(0,\La^*)$, problem \eqref{pq} has at least two solutions in {$H_0^1(\Om)\cap C^+_{\varphi_q}(\Omega)$}.
		\end{theorem}


		\section{Variational framework}\label{sec2}
		\noi	For $u:\Om\to\mathbb R$, measurable function, and  for $1\leq p \leq +\infty$, we define the Lebesgue space $L^{p}(\Om)$ as 
		$$L^{p}(\Om)=\{u:\Om\to\mathbb R \;\; \text{measurable}| \int_\Om| u|^p ~dx<+\infty\}$$ equipped with the usual norm denoted by $\Vert u\Vert_{p}$.
		Now the Sobolev space $H^{1}_0(\Om)$ is defined as 
		$$H^{1}_0(\Om)=\{u\in L^2(\Om)| \int_\Om|\nabla u|^2 ~dx<+\infty\}$$   endowed with  the norm
		$$\|u\|:=\displaystyle \left({\int_{\Om} |\nabla u|^2 ~dx}\right)^{1/2}.$$ Since $\Om\subset \mb R^2$ is a smooth bounded domain, we  have the continuous embedding $H^1_0(\Om)\hookrightarrow L^p(\Om)$ for $p\in [1,+\infty)$.
		Moreover, the embedding $H^{1}_0(\Om)\ni u \mapsto \exp(|u|^\beta)\in L^1(\Om)$ is compact for all $\beta \in \left[1,2\right)$ and is continuous for $\beta = 2$. Consequently, the map $\mc M: H^{1}_0(\Om) \to L^p(\Om)$ for $p \in [1,+\infty)$, defined by $\mc M(u):= \exp\left( |u|^{2}\right)$ is continuous with respect to the $H^1_0-$norm topology.
		
		
		\noi	The  natural energy functional  associated to  problem \eqref{pq} is the following:
		\begin{equation}\label{en}
			I_\la(u)=\left\{
			\begin{array}{l}
				\frac{1}{2}\displaystyle\int_{\Om }(1+2|u|^{2})|\nabla u|^{2}~dx-\la\frac{1}{1-q}\int_\Om \al(x) u^{1-q} ~dx-\la\int_{\Om} {F(x,u(x))}~dx~~\mbox{\;\; if}~~ q\not =1;\\
				\frac{1}{2}\displaystyle\int_{\Om }(1+2|u|^{2})|\nabla u|^{2}~dx-\la\int_\Om \al(x) \log |u| ~dx- \la\int_{\Om} {F(x,u(x))}~dx~~\mbox{\;\;\;\;\;\;\; if}~~ q =1.
			\end{array}
			\right.
		\end{equation}
		
		\noi	Observe that, the functional $I_\la$ is not well defined in $H^{1}_0(\Om)$ because of the nature of the singularity $(0<q<3)$ as well as, due to  the fact that $\displaystyle\int_{\Om} u^{2}|\nabla u|^{2}dx$ is not finite for all $u\in H_0^{1}(\Om) $. So, it is difficult to apply variational methods directly in our problem \eqref{pq}. In order to get rid of this inconvenience, first we apply the following change of variables introduced in \cite{CJ}, namely, $w:=h^{-1}(u),$ where $h$ is defined by
		\begin{equation}\label{g}
			\left\{
			\begin{array}{l}
				h^{\prime}(s)=\displaystyle\frac{1}{\left(1+2|h(s)|^{2}\right)^{\frac{1}{2}}}~~\mbox{in}~~ [0,+\infty),\\
				h(s)=-h(-s)~~\mbox{in}~~ (-\infty,0].
			\end{array}
			\right.
		\end{equation}
		
		\noi	Now we gather some properties of $h$, which we follow throughout in this article. For the detailed proofs of such results, one can see  \cite{CJ,do}.
		\begin{lemma}\label{L1} The function $h$ satisfies the following properties:
			\begin{itemize}
				\item[$(h_1)$] $h$ is uniquely defined, $C^{\infty}$ and invertible;
				\item[$(h_2)$] $h(0)=0$;
				\item[$(h_3)$] $0<h^{\prime}(s)\leq 1$ for all $s\in \mathbb{R}$;
				\item[$(h_4)$] $\frac{1}{2}h(s)\leq sh^{\prime}(s)\leq h(s)$ for all $s>0$;
				\item[$(h_5)$] $|h(s)|\leq |s|$ for all $s\in \mathbb{R}$;
				\item[$(h_6)$] $|h(s)|\leq 2^{1/4}|s|^{1/2}$ for all $s\in \mathbb{R}$;
				\item[$(h_7)$] $\ds \lim_{s\to+\infty}{h(s)}/{s^{\frac 12}}=2^{\frac{1}{4}}$;
				\item[$(h_8)$]  $|h(s)|\geq h(1)|s|$ for $|s|\leq 1$ and $|h(s)|\geq h(1)|s|^{1/2}$ for $|s|\geq 1$;
				\item[$(h_9)$] $h''(s)=-2h(s)(h'(s))^4,$ for all $s\geq0$ and $h^{\prime \prime}(s)<0$ when $s>0$, $h^{\prime \prime}(s)>0$ when $s<0$.
				\item[$(h_{10})$] the function $\frac{(h(s))^{\gamma}h'(s)}{s}$ is strictly increasing  for $\gamma\geq3$;
				\item[$(h_{11})$] $\ds\lim_{s\to 0} {h(s)}/{s}=1$;
				\item[$(h_{12})$] $|h(s)h^{\prime}(s)|<1/ \sqrt[]{2}$ for all $s\in \mathbb{R}$;
				\item[$(h_{13})$] the function $h(s)^{-\gamma}h'(s)$ is decreasing for all $s>0$, where $\gamma>0$.
			\end{itemize}
			
		\end{lemma}
		\noi	After employing the change of variable  $w=h^{-1}(u)$ in \eqref{en}, we define the new functional $J_\la:H^{1}_0(\Om)\to \mb R$ as
		\begin{equation}\label{energy}
			J_\la(w)=\left\{
			\begin{array}{l}
				\frac{1}{2}\displaystyle\int_{\Om} |\nabla w|^{2}dx-\la\frac{1}{1-q}\int_\Om \al (x)|h(w)|^{1-q}dx-\la\int_{\Om}F(x,h(w))dx~~\mbox{if}~~ q\not =1;\\
				\frac{1}{2}\displaystyle\int_{\Om} |\nabla w|^{2}dx-\la\int_\Om \al (x)\log|h(w)|dx-\la\int_{\Om}F(x,h(w))dx~~\mbox{if}~~ q =1.
			\end{array}
			\right.
		\end{equation}
		\begin{remark}  The functional $J_\la(w)$ is well defined for $w\in H^1_0(\Om)\cap C_{\varphi_q}^+(\Om)$. Indeed,  let us define the set $\mc D(J_\la):= \{w\in H^1_0(\Om):\; J_\la(w)< +\infty\}$. Now we intend to show that this set is non empty.
			Let  $1<q< 3$. Then using the fact that $h(s)$ is increasing in $s>0$, $\varphi_{1,\Om}\in C^+_\de(\Om)$ and  Lemma \ref{L1}-$(h_8)$, we have
			\begin{equation*}
				h(w)> h(C_1\varphi_{1,\Om}^{\frac{2}{q+1}})>h(C_2\de ^{\frac{2}{q+1}})>
				\left\{
				\begin{array}{l}
					C_2h(1)\de^{\frac{2}{q+1}} \mbox{\;\; if }C_2\de^{\frac{2}{q+1}}<1,\\
					h(1)\mbox{\;\; if }C_2\de^{\frac{2}{q+1}}\geq 1,
				\end{array}
				\right. 
			\end{equation*}  where $C_2$ is a positive constant. Thus, for any $\psi \in H_0^1(\Om)$ and $w \in C_{\varphi_q}^+(\Om)$, we deduce the following by applying Lemma \ref{L1}-$(h_3)$, H\"older's inequality,  Hardy's inequality,  and the Sobolev embedding:
			\begin{align}\label{gm-11}
				\int_{\Om}h(w)^{-q}h'(w)\psi dx =&	\int_{\Om\cap \{x\,:\,C_2\de(x)^{\frac{2}{q+1}}<1 \}}h(w)^{-q}h'(w)\psi dx +	\int_{\Om\cap \{x\,:\,C_2\de(x)^{\frac{2}{q+1}}\geq1 \}}h(w)^{-q}h'(w)\psi dx \notag\\&\leq C_3 \left(\int_\Om \frac{dx}{\de^{\frac{2(q-1)}{(q+1)}}}\right)^{\frac12}\left(\int_\Om\frac{\psi^2 }{\de^{2}}~dx\right)^{\frac12}+ C_4 h(1)^{-q}\int_\Om \psi~ dx\notag\\
				&< C_5 \|\psi\|<+\infty,
			\end{align} since $\frac{2(q-1)}{q+1}<1$ for $1<q<3$, where $C_3,\, C_4,\, C_5$ are positive constants. 
			For the  case $0<q\leq1$,  arguing is a similar manner as above and following \cite{haito,yang-singular}, we get \eqref{gm-11}. In view of this, we obtain that the set $\mc D(J_\la)\neq \emptyset$ for $0<q<3$. 
		\end{remark}
		\noi Next, we check for the Gateaux differentiability of the functional $J_\la.$   For $0<q<1$ and $w\in H^1_0(\Om)$ with $w\geq c \de,$ using the idea of \cite{haito} combining with the properties of $h$ described in Lemma \ref{L1}, we can show that the functional $J_\la$ is Gateaux differentiable  at $w$. For the range $1\leq q<3$, we have the following lemma regarding the similar property of $J_\la$.
		\begin{lemma}\label{diff}
			Let $\mc S:=\{w\in H^1_0(\Om) : w_1\leq w\leq w_2\},$ where $w_1\in C^+_{\varphi_q}(\Om)$  and $w_2\in H^1_0(\Om)$. Then $J_\la$ is Gateaux differentiable at $w$ in the direction $v-w$ for $v,\, w\in \mc S.$
		\end{lemma}
		\begin{proof}
			We have to  prove that
			\begin{align*}
				\lim_{t \to 0} \frac{J_\la(w+t(v-w))-J_\la(w)}{t} &= \int_\Om \nabla w\nabla(v-w)dx
				-\la \int_\Om \al (x)h(w)^{-q}h'(w)(v-w)~dx\\ &\qquad\qquad-\la\int_{\Om}f(x,h(w))h'(w)(v-w)~dx.
			\end{align*}
			For the first term  and third term in the right hand side of the  above expression, the proof follows in a standard way. So,  we are left to show  for the  second term, that is, for the singular term. Since $\mc S$ is a convex set,  for any $t \in (0,1)$, $w+t(v-w)\in \mc S$. Let us define the function $H:H^1_0(\Om)\to \mb R$ as
			\begin{equation*} H(w)=\left\{
				\begin{array}{l} \displaystyle\frac{1}{1-q}\int_{\Om}h(w)^{1-q}~dx \quad\mbox { if } q\not=1\\ 
					\displaystyle\int_{\Om}\log (h(w))~dx\quad \mbox{ if } q=1.
				\end{array}\right.
			\end{equation*}
			Now using mean value theorem, for any $0<q<3$, it follows that
			\begin{align}\label{1}
				\frac{H(w+t(v-w))-H(w)}{t} 
				&= \int_\Om h(w+t\theta(v-w))^{-q} h'(w+t\theta(v-w)) (v-w)~dx
			\end{align}
			for some $\theta \in (0,1)$. Since  $w+t\theta(v-w)\in \mc S$,  using  Lemma \ref{L1}-$(h_{13})$ ,  it follows that
			\[h(w+t\theta(v-w))^{-q}h'(w+t\theta(v-w))(v-w)~dx  \leq  h(w_1)^{-q}h'(w_1)(v-w)~dx .\]
			Now recalling  \eqref{gm-11}, we get $$\int_{\Om} h(w_1)^{-q}h'(w_1)(v-w)~dx < +\infty.$$
			Thus, applying the Lebesgue dominated convergence theorem and passing to the limit $t \to 0^+$ in \eqref{1}, we have
			\begin{align}\label{kk}\lim_{t\to 0}\frac{H(w+t(v-w))-H(w)}{t} = \int_\Om h(w)^{-q}h'(w)(v-w)~dx.
			\end{align}  
			This completes the proof.
		\end{proof}

		\noi	Thus, using the properties of the functions $f,h$  and the above results, it can be derived  that \eqref{energy} is the associated energy functional to the following problem:
		\begin{equation}
			\label{pp}\left\{
			\begin{array}{rlll}
				-\Delta w&=&\la\left( \al(x) h(w)^{-q} h'(w)+
				f(x,h(w))h'(w)\right) \; \text{in}\;
				\Om,\\
				w&>& 0~\mbox{in}~ \Om,\\
				w&=&0 ~\mbox{on}~ \partial\Om.
			\end{array}
			\right.
		\end{equation} Moreover, applying Lemma \ref{L1} and following the idea as in \cite[Proposition 2.2 ]{severo}, one can show that 
		if $w$ is a solution to \eqref{pp}, 
		then  $u = h(w)$ is a solution to the problem \eqref{pq}. Thus, our main objective is now reduced to proving the existence of  solutions to the new transformed equation  \eqref{pp}.
		\begin{definition}
			A function $w \in H^{1}_0(\Om)$ is  said to be a weak solution to \eqref{pp} if for every compact set $\mc K\subset \Om$, there exists a constant $m_{\mc K}>0$ such that $w>m_{\mc K}$ holds in $\mc K$ and  for every $\phi \in H^{1}_0(\Om)$, we have
			\begin{align}\label{wk}
				\int_{\Om} \nabla w\nabla \phi ~dx-\la\int_\Om \al(x)h(w)^{-q} h'(w)\phi ~dx-\la\int_{\Om}f(x,h(w))h'(w)\phi ~dx=0.\end{align} 
		\end{definition}\noi In the next lemma, we discuss some comparison type result related to our problem \eqref{pp}.
		\begin{lemma}\label{comp-princ}
			Let {$w_1,w_2 \in H^1_0(\Om)\cap C^+_{\varphi_q}(\Omega)$} satisfy \begin{align*}-\De w_1&\leq \gamma(x)h(w_1)^{-q}h'(w_1), \;\;\; x\in\Om;\\
				-\De w_2&\geq \gamma(x)h(w_2)^{-q}h'(w_2), \;\;\; x\in\Om,
			\end{align*} where $\ga\in L^\infty(\Om)$ with $\ga>0$.
			Then $w_1\leq w_2$ a.e. in $\Om$. 
		\end{lemma}
		\begin{proof}
			The proof of this lemma follows in a similar fashion as in \cite [Lemma 2.2]{yang-singular}.
		\end{proof}

		\textbf{Notations.}  In the next subsequent sections, we make use of the following notations:
		\begin{itemize}
			\item If $u$ is a measurable function, we denote  the positive  and negative parts by $u^{+}=\max\left\{u,0\right\}$ and $u^{-}=\max\left\{-u,0\right\}$, respectively.
			\item For any function $f$, supp $f=\{x\,:\, f(x)\not=0\}$. 
			\item If $A$ is a measurable set in $\mathbb{R}^{2}$, we denote the Lebesgue measure of $A$  by $\vert A \vert$ . 
			\item The arrows $\rightharpoonup $ , $\to $ denote weak convergence,  strong convergence, respectively.
			\item The arrow $\hookrightarrow $ denotes continuous embedding.
			\item $B_r$ denotes the ball of radius $r>0$ centered at $0\in H^1_0(\Om).$
			\item  $\overline{B_r}$ denotes the closure of the ball $B_r$ with respect to $H_0^{1}(\Om)$-norm topology.
			\item  ${\partial B_r}$ denotes the boundary of the ball $B_r$.
			\item $B_r(x)$ denotes the ball of radius $r>0$ centered at $x\in H^1_0(\Om).$
			\item For any $p>1$, $p':=\frac{p}{p-1}$ denotes the conjugate of $p$.
			\item $c,C_0, C_{1},C_{2},\cdots, \tilde C_1, \tilde C_2,\cdots, C$ and $\tilde C$ denote positive constants which may vary from line to line.
		\end{itemize}
		
		\section{Proof of Theorem \ref{thm1} : Existence and non-existence results}
		In this section,  to prove the existence and non-existence results for the problem  \eqref{pp}, we first  need to study the existence and regularity result for the following purely singular problem:
		\begin{equation}
			\label{sp}\left\{
			\begin{array}{rlll}
				-\Delta w&=&\la \al(x) h(w)^{-q} h'(w)\; \text{in}\;
				\Om,\\
				w&>&0 \; \text{in}\;
				\Om,\\
				w&=&0 ~\mbox{on}~ \partial\Om.
			\end{array}
			\right.
		\end{equation}
		Now, we have the following result for the problem \eqref{sp}. 
		\begin{theorem}\label{tsp}
			Let $\Om\subset \mb R^2$ be a smooth bounded domain.  Assume that $\la>0,$ $q>0$, $\al$ satisfies  $(\al_1)$ and   the function $h$ is defined  in \eqref{g}. Then, 
			\begin{enumerate}
				\item[(i)]the  problem \eqref{sp} has a unique solution for each $\la>0$,  say  $\underline{w_\la}$, in $H_0^1(\Omega)\cap C^+_{\varphi_q}(\Omega),$  for $q<3$;
				\item[(ii)]  the solution $\underline {w_\la}\in C^1(\overline\Omega)$ if $q<1$,  $\underline {w_\la}\in C^{1-\e}(\overline\Omega)$ for any small $\e>0$ if $q=1$ and  $\underline{w_\la} \in C(\overline\Omega)$ if $q<3$;
				\item[(iii)] the map $\la \to  \underline {w_\la}$  is non-decreasing and continuous from $\mathbb R^+$ to $C(\overline\Omega)$.
			\end{enumerate}
		\end{theorem}
		\begin{proof} Let us set	$\rho(s):=\la \al(x) h(s)^{-q}h'(s)$. Then $\rho(w)$ verifies  the hypotheses of  Proposition 4.1 in \cite{adijg}. Hence, using \cite[Proposition 4.1]{adijg}, we can infer that there exists a unique solution to \eqref{sp}, say $\underline{w_\la}$, such that $\underline{w_\la}\in H^1_0(\Om)$ for $1<q<3$. 
			Next,  following the proof of  \cite[Theorem 2.2]{crandal },   there exist two positive constants $c_1:=c_1(\la,q)<<1,$ and $ c_2:=c_2(\la,q)$ such that
			\begin{align}
				c_1(q,\la)\de (x)&\leq \uline{w_{\la}}\leq c_2(q,\la)\de(x) \text {\;\;\; if \;}  0<q<1,\label{inqq1}\\
				c_1(q,\la)\de (x)^\frac{2}{q+1}&\leq \uline{w_{\la}}\leq c_2(q,\la)\de(x)^\frac{2}{q+1} \text {\;\;\; if \;}  1<q<3.\label{inqq2}
			\end{align} Furthermore, for the case $q=1$, again recalling \cite[Theorem 2.2]{crandal }, we  can find that there exists a constant $c´(\la)>0$ and for any $\e>0$ small enough, there exists a  constant $c_\e(\la)>0$ such that 
			\begin{align}\label{inqq3}
				c´(\la)\de (x)&\leq \uline{w_{\la}}\leq c_\e(\la)\de(x)^{1-\e}.
			\end{align} This, in combination with standard elliptic regularity theory, implies that $\uline{w_{\la}}\in C^+_{\varphi_q}(\Omega)$. Thus, $(i)$ follows.\\
			Again, using \cite[Proposition 4.1]{adijg} (also see \cite[Theorem 2.2]{crandal }), we get $(ii)$.\\
			{	 Finally, using $(i)-(ii)$ and the maximum principle, we obtain $(iii)$. This completes the proof.}\\
		\end{proof}
		\begin{remark}
			One can check that \eqref{sp} is the transformed form (with the transformation $u=h(w)$) of the following purely singular problem with the modified quasilinear operator corresponding to the problem \eqref{pq} :
			\begin{equation}
				\label{s1p}\left\{
				\begin{array}{rlll}
					-\Delta u	-\Delta(u)^2 u&=&\la \al(x) u^{-q} \; \text{in}\;
					\Om,\\
					u&>&0 \; \text{in}\;
					\Om,\\
					u&=&0 ~\mbox{on}~ \partial\Om,
				\end{array}
				\right.
			\end{equation}  where $\la,\,\al,\,\Om$ are as  in Theorem \ref{tsp}. So, by the properties of $h$ and following the idea of the proof of the  Proposition 2.2  in \cite {severo}, we can deduce that \eqref{s1p} has a solution $h(\uline{w_\la})$ for every $\la>0$, which satisfies all the properties in Theorem \ref{tsp}.
		\end{remark}
		\noi	From the assumption $(f2)$ and \eqref{rmk1}, we obtain that for any $\e>0$, $r\geq 1$,  there exist $\tilde C(\e)$ and $C(\e)>0$ such that
		\begin{align} 
			&|f(x,s)| \le \e |s|+ \tilde C(\e) |s|^{r-1} \exp\left((1+\e)|s|^{4}\right)\;\; \text{for all}\; (x,s)\in \Om \times \mb R,\label{f}\\
			&|F(x,s)| \le \e |s|^2 + C(\e) |s|^r \exp\left((1+\e)|s|^{4}\right)\;\; \text{for all}\; (x,s)\in \Om \times \mb R.\label{F}
		\end{align}
		
		For any $w \in H^{1}_0(\Om)$, in light of the Sobolev embedding, we have $w \in L^q(\Om)$ for all $q \in [1,+\infty)$.\\
		Let us define the set $\mathcal Q$ as 
		\[\mc Q=\{\la>0\;:\; \text { the problem \eqref{pp} has a weak solution in } H^1_0(\Om)\}\] and let $\La^*:=\sup \mc Q.$ Then we have the following result:
		
		\begin{lemma}\label{lem1}
			Assume that the conditions in Theorem \ref{thm1} hold and let $h$ be defined as in \eqref{g}. Then the set $\mc Q$ is non empty.
		\end{lemma}
		\begin{proof}
			First, we consider the case $0<q<1.$\\
			Using \eqref{F}, Lemma \ref{L1}-$(h_5), (h_6)$ with the Sobolev embedding and  H\"older's inequality, from   \eqref{energy}, we get
			\begin{align}
				J_\la(w)&\geq\frac 12 \|w\|^2-\frac{\la}{q-1}\int_\Om \al (x)h(w)^{1-q}~dx-\la\e\int_\Om h(w)^2 ~dx-\la C(\e)\int_\Om |w|^r\exp((1+\e)h(w)^4)~dx\notag\\
				&\geq \frac 12 \|w\|^2-\frac{\la \|\al\|_\infty}{q-1} \int_\Om |w|^{1-q}~dx-\la\e\int_\Om w^2 ~dx-\la C(\e)\int_\Om |w|^r\exp(2(1+\e)w^2)~dx\notag\\
				&  \geq \frac 12 \|w\|^2-\frac{\la \|\al\|_\infty}{q-1} C_1 \| w\|^{1-q}-\la\e C_2\| w\|^2-\la C_3(\e)\|w\|_{rp'}^{r}\left(\int_\Om\exp(2p(1+\e){w^2})~dx\right)^{1/p}\label{mon}\\
				&  \geq \left(\frac 12-\la \e C_2\right) \|w\|^2-\frac{\la \|\al\|_\infty}{q-1} C_1 \| w\|^{1-q}-\la C_4(\e)\|w\|^{r}\left(\int_\Om\exp\left(2p(1+\e)\|w\|^2.\frac{w^2}{\|w\|^2}\right) ~dx\right)^{1/p}, \label{a1}
			\end{align} for any $\e>0,\,p>1$ and $r>2.$ Choose $\|w\|=r_0$ with $0<r_0<1$ sufficiently small, $0<\e< \frac{1}{2C_2}$ sufficiently small  and $p>1$ very near to $1$ such that $2(1+\e)pr_0<4\pi$. Then by using Theorem \ref{TM-ineq},  from \eqref{F} 
			\eqref{a1}, we obtain 
			\begin{align*}
				&\frac 12 \|w\|^2-\la\int_\Om F(x,h(w))~dx\geq 2\delta_0 \;\;\text{ for all } w\in \partial B_{r_0};\\
				&\frac 12 \|w\|^2-\la\int_\Om F(x,h(w))~dx\geq 0 \;\;\text{ for all } w\in B_{r_0}.
			\end{align*} Now we can choose $\la=\la_0>0$ sufficiently small so that the last  two relations  yield that
			\begin{align*}
				J_{\la_0}|_{\partial B_{r_0}}\geq \de_0>0.
			\end{align*}
			Set $m_0:=\ds\inf_{w\in B_{r_0}} J_{\la_0}(w)$. Since  for $t>0$ very small and $w\not=0$, from \eqref{energy}, we have \[J_{\la_0}(tw)\leq \frac 12 \|tw\|^2-C\la_0\|tw\|^{1-q},\] which implies that $m_0<0.$ Let $\{w_k\}\subset B_{r_0}$ be a minimizing sequence such that  as $k\to+\infty$,
			\begin{align*}
				&J_{\la_0}(w_k)\to m_0;\\
				&w_k\rightharpoonup w_0 \text{\;\;  weakly in } H^1_0(\Om);\\
				& w_k\to w_0 \text{\;\;  strongly in } L^{p}(\Om),\,\, p\geq1\text{ and }\\ & w_k(x)\to w_0(x) \text{  point-wise a.e. in } \Om.
			\end{align*} Now, without loss of generality, let us assume that $w_k\geq 0$ due to the fact that $J_\la(w)=J_\la(|w|)$. Then, using the Sobolev and the Hölder´s inequality,  one can easily deduce that  for $0<q<1$, 
			\begin{align}\label{hs}
				\int_\Om w_k^{1-q} dx\to \int_\Om w_0^{1-q}dx \,\text{ and }
				\int_\Om |w_k-w_0|^{1-q} dx\to 0\,\,\text{ as $k\to+\infty$.}
			\end{align}

			\noi			Next, by the mean value theorem, there exists $\tilde w_k$ in between $w_0$ and $w_k$ such that using $|h'(\tilde w_k)|\leq 1$ and \eqref{hs}, we deduce
			\begin{align*}
				\int_\Om\al(x)h(w_k)^{1-q} ~dx&\leq \int_\Om\al(x)h(w_0)^{1-q} ~dx+\int_\Om\al(x)|h(w_k)-h(w_0)|^{1-q} ~dx\\
				& \leq \int_\Om\al(x)h(w_0)^{1-q} ~dx+\int_\Om\al(x)|w_k-w_0|^{1-q}|h'(\tilde w_k)|^{1-q} ~dx\\
				&\leq \int_\Om\al(x)h(w_0)^{1-q} ~dx+\|\al\|_\infty\int_\Om|w_k-w_0|^{1-q}~dx\\
				&\leq \int_\Om\al(x)h(w_0)^{1-q} ~dx+o(1),\end{align*}
			as $k\to +\infty$. Similarly, we get $\ds \int_\Om\al(x)h(w_0)^{1-q} ~dx\leq \int_\Om\al(x)h(w_k)^{1-q} ~dx+o(1)$ as $k\to+\infty.$ Therefore, from the last two relations, we obtain
			\begin{align}\label{a5}
				\int_\Om\al(x)h(w_k)^{1-q} ~dx\to \int_\Om\al(x)h(w_0)^{1-q} ~dx\text{ as } k\to+\infty.
			\end{align} Next, using Lemma \ref{L1}-$(h4)$,  \eqref{f} and then borrowing the similar argument as in  the third  and fourth terms in  \eqref{mon}, for sufficiently small $\e>0$ and for any $p>1$, we deduce
			\begin{align*}
				\int_\Om  f(x,h(w_k))h^\prime(w_k) w_k  ~dx&\leq 	\int_\Om  f(x,h(w_k)) h(w_k)~dx\notag\\ 
				&\leq \int_\Om  \left(\e |h(w_k)|^2+ \tilde C(\e) |h(w_k)|^{r} \exp\left((1+\e)|h(w_k)|^{4}\right)\right) dx\notag\\
				&\leq \e \tilde C_2 \|w_k\|^2+\tilde C_4(\e)\|w_k\|^{r}\left(\int_\Om\exp\left(2p(1+\e)\|w_k\|^2.\frac{w_k^2}{\|w_k\|^2}\right) ~dx\right)^{1/p}. 
			\end{align*}In the last relation, using Theorem \ref{TM-ineq}, with sufficiently small $\|w_k\|<r_0<<1$ and $p>1$ very close to $1$ so that $2(1+\e)pr_0<4\pi$, we get 
			\begin{align}\label{F22}
				\limsup_{k\to+\infty}	\int_\Om  f(x,h(w_k))h^\prime(w_k) w_k  ~dx<+\infty.
			\end{align}
			Now using Lemma \ref{L1}-$(h_8)$ and  \eqref{F22}, for some large   $N (>>1)\in\mathbb N$, we  deduce
			\begin{align*}\int_{\Om\cap\{x\,:\,h(w_k)(x)>N\}} f(x,h(w_k)) ~dx&\leq \frac 1N \int_{\Om\cap\{x\,:\,h(w_k)(x)>N\}} f(x,h(w_k)) h(w_k) ~dx\\&\leq \frac 2N \int_{\Om\cap\{x\,:\,h(w_k)(x)>N\}}  f(x,h(w_k)) h'(w_k) w_k ~dx\\&=O\left(\frac 1N\right).
			\end{align*} The last relation, together with the Lebesgue dominated convergence theorem, implies that
			\begin{align}\label{FoF}\int_{\Om} f(x,h(w_k)) ~dx&=\int_{\Om\cap\{x\,:\,h(w_k)(x)\leq N\}} f(x,h(w_k)) ~dx+\int_{\Om\cap\{x\,:\,h(w_k)(x)>N\}} f(x,h(w_k)) ~dx\notag\\
				&=\int_{\Om\cap\{x\,:\,h(w_k)(x)\leq N\}} f(x,h(w_k)) ~dx+O\left(\frac 1N\right)\notag\\
				& \to\int_{\Om} f(x,h(w_0)) ~dx, \,\;\; \text{as $k\to+\infty$ and $N\to+\infty$.}
			\end{align} 
			Since by $(f4)$ and \eqref{FoF},
			$F(x,h(w_k))\leq M_1(1+f(x,h(w_k)))\in L^1(\Om)$, for all $ k\in\mb N$,  using the  Lebesgue dominated convergence theorem, 
			we obtain
			\begin{align}\label{a6}
				\int_\Om F(x,h(w_k)) ~dx\to \int_\Om F(x,h(w_0)) ~dx\; \text{\;\;\; as } k\to +\infty.
			\end{align}
			Now  from the weak lower semi-continuity of the norm, we have
			\[r_0\geq \liminf_{k\to+\infty}\|w_k\|\geq \|w_0\|,\] which yields that $w_0\in B_{r_0}.$ Therefore,\[J_{\la_0}(w_0)\geq m_0.\] Furthermore, recalling \eqref{a5} and \eqref{a6}, we get
			\[m_0=\lim_{k\to+\infty} J_{\la_0}(w_k)\geq J_{\la_0}(w_0)\geq m_0.\] Thus, \[J_{\la_0}(w_0)=m_0<0.\] This yields that $w_0 (\not\equiv 0)$ is a local minimizer of $J_{\la_0}$ in $H_0^1(\Om).$ \\
			Next, we claim that $w_0$ is a weak solution to the problem \eqref{pp}. Note that, for any $\phi\geq0, \phi \in H^1_0(\Om)$,
			\begin{align}\label{a7}
				\liminf_{t\to 0^+} \frac{J_{\la_0}(w_0+t\phi)-J_{\la_0}(w_0)}{t}\geq 0.
			\end{align}
			It can be derived from the last expression that $ -\De w_0\geq 0$ in $ \Om$ in the weak sense and hence, by the strong maximum principle, $w_0>0$ in $\Om.$ Furthermore, employing Fatou's lemma in \eqref{a7}, we infer that
			\begin{align}\label{a8}
				\int_\Om \nabla w_0\nabla \psi ~dx\geq \la \int_\Om \al(x)h(w_0)^{-q} h'(w_0)\psi ~dx+\la\int_\Om f(x,h(w_0)) h'(w_0)\psi ~dx \;\;\;\; \text{ for all } \psi\in H^1_0(\Om),\; \psi\geq 0.
			\end{align} Now for any $\phi\in H^1_0(\Om)$ and $\e>0$, taking $\psi=(w_0+\e\phi)^+$ as a test function in \eqref{a7} and dividing it by $\e>0$, we obtain 
			\begin{align*}
				\int_\Om \nabla w_0\nabla \phi ~dx+\e\int_\Om |\nabla \phi|^2 ~dx&\geq \frac {\la }{\e}\int_\Om \al(x)(h(w_0+\e\phi)^{1-q} -h(w_0)^{1-q})  ~dx\notag\\&\qquad\qquad\qquad+\frac {\la}{\e}\int_\Om ( F(x,h(w_0+\e\phi))- F (x,h(w_0))) ~dx.
			\end{align*} Letting the limit $\e\to 0^+$ in the last expression, we deduce
			\begin{align*}
				\int_\Om \nabla w_0\nabla \phi ~dx\geq  {\la }\int_\Om \al(x)h(w_0)^{-q}h'(w_0) \phi  ~dx+\la\int_\Om f (x,h(w_0)) h'(w_0) \phi ~dx.
			\end{align*} Again, by taking  $-\phi$ in place of $\phi$, we get the reverse inequality in the last relation. Therefore, $w_0$ is a weak solution to the problem \eqref{pp}. \\
			Next, we discuss the case $1\leq q<3.$  For that, we consider the following problem: \begin{equation}
				\label{pp,}\left\{
				\begin{array}{rlll}
					-\Delta w&=&\la \al(x) (h(w)+\e)^{-q} h'(w)+
					\la f(x,h(w))h'(w) \; \text{in}\;
					\Om,\\
					w&>&0 \; \text{in}\;
					\Om,\\
					w&=&0 ~\mbox{on}~ \partial\Om,
				\end{array}
				\right.
			\end{equation}
			where $0<\e<1$ is sufficiently small. 
			We show the  existence of solution to \eqref{pp,} by constructing a sub-solution and a super-solution to \eqref{pp,}.
			Let $W$ be  a  solution to
			\begin{align*}
				-\De W=1 \text{   in } \Om;\;\;\;\; W=0 \text{  on } \partial \Om.
			\end{align*} Then by the maximum principle, $W>0$ in $\Om$ and by the standard elliptic regularity theory, $W\in C^1(\overline\Om)$ and hence, $W$ is bounded on $\oline\Om$.
			\noi Set $$\overline w :=\underline{w_\la}+ MW,$$
			where 
			$M>0$ is a sufficiently large real constant and $\underline{w_\la}$ is  the  solution to \eqref{sp}. Therefore,
			\begin{align}\label{a9}
				-\De \overline w=\la \al(x) h(\underline{w_\la})^{-q} h'(\underline{w_\la})+M.
			\end{align}
			So, using Lemma \ref{L1}-$(h_3),(h_6)$ combining with $(f2)$ and the continuity of $f$, for any $\e>0$ there exists some constant $c(\e,\Om)$,  such that 
			\begin{align}\label{a12}
				&\la f\left(x,h\left(	\underline{w_\la}+ MW\right)\right)h'\left(	\underline{w_\la}+ MW\right)\notag\\&<\la c\exp\left((1+\e)h\left(	\underline{w_\la}+ MW\right)^4\right)\notag\\
				&<\la c\exp\left(2(1+\e)\left(	\underline{w_\la}+ MW\right)^2\right)\notag\\
				& <\la c\exp\left(2(1+\e)\left(	\|\underline{w_\la}\|_\infty+ M\|W\|_\infty\right)^2\right)\notag\\
				&=\la C(M)\notag\\
				&<M \text{ \;\; for sufficiently small } \la>0.
			\end{align} Plugging  \eqref{a12} in \eqref{a9}, we infer that $\overline w$ is a super-solution to \eqref{pp,}. Now let us set $$\underline w:=m W$$ for some sufficiently small constant $m=m(\la)>0$ such that $\uline w<1$.  Therefore, \[-\De \underline w=m.\] We claim that $\underline w$ is a sub-solution to \eqref{pp,}. To prove the claim, it is enough to show 
			\begin{align}\label{a13}
				m<\frac{\la \al_0h'(mW)}{(h(mW)+\e)^q}.
			\end{align} In virtue of Lemma \ref{L1}-$(h_4), (h_5), (h_8)$, we deduce
			\begin{align*}
				\frac{\la\al_0}{m(h(mW)+\e)^q} h'(mW)&\geq \frac{\la\al_0}{m(h(mW)+\e)^q} \frac{ h(mW)}{2mW}\notag\\
				&\geq \frac{\la\al_0}{m(mW+1)^q} \frac{ h(1)}{2}\notag\\
				&\geq {\frac{\la\al_0}{m^{q+1}(\|W\|_\infty+1/m)^q} \frac{ h(1)}{2}>1},
			\end{align*} for sufficiently small $0<m<1$. This establishes the claim. \\Now for $\la>0$ sufficiently small,  choosing $M$ sufficiently large and $m$ sufficiently small so that $M>>m$, from \eqref{a12} and \eqref{a13}, we obtain $0<\underline w\leq \overline w$. Hence, there is a solution {$\underline w\leq w_\e\leq \overline w$} to \eqref{pp,} in $H^1_0(\Om)$. Since $\underline w,\, \overline w$ do not depend on $\e$, $$w_\e(x)\to w_\la(x)\text{ \;\; point-wise in }  \Om \text { as } \e\to 0^+.$$ Next, we will show that $w_\la$ is a weak solution to \eqref{pp}.   From Lemma \ref{L1}-$(h_{4})$, $(h_5)$, $(h_8)$, we deduce
			\begin{align}\label{st}\al(x)(h(w_\e)+\e)^{-q}h'(w_\e)w_\e\leq \al(x)(h(w_\e)+\e)^{-q}h(w_\e)&\leq \al(x)h(\underline w)^{1-q}\notag\\ & \leq \left\{
				\begin{array}{l}
					\|\al\|_\infty \underline w^{1-q},~~\mbox{if}~~ 0\leq q\leq 1,\\
					\|\al\|_\infty h(1)^{1-q}\underline w^{1-q},~~\mbox{if}~~ 1\leq q\leq 3.
									\end{array}
				\right.
			\end{align} On the other hand,  by $(f1)$ and Lemma \ref{L1}-$(h_{10})$, it follows that  $f(x,h(s))h'(s)=\frac{f(x,h(s))}{h(s)^3}.\frac{ h(s)^3h'(s)}{s}.s$ is increasing in $s>0$. Using this together with  Lemma \ref{L1}-$(h_4)$, we find
			\begin{align}\label{aije}f(x,h(w_\e))h'(w_\e)w_\e \leq f(x,h(\oline w))h(\oline w).\end{align}
			Testing \eqref{pp,} against the test function $w_\e$ and then using \eqref{st}, \eqref{aije}, we get
			\begin{align*}\int_\Om |\nabla w_\e|^2 ~dx&=\la \int_\Om \al(x)(h(w_\e)+\e)^{-q}h'(w_\e)  w_\e ~dx+ \la\int_\Om f(x,h(w_\e))h'(w_\e) w_\e ~dx\\
				&\leq \la  C(\al) \int_\Om \underline w^{1-q} ~dx+ \la\int_\Om f(x,h(\oline w))h(\oline w)  ~dx<+\infty.
			\end{align*} { Thus, $\{w_\e\}$ is  bounded in $H^1_0(\Om)$. So, up to some sub-sequence, $w_\e\rightharpoonup w_\la$ in $H^1_0(\Om)$ as $\e\to 0$. 
			} Again, testing  \eqref{pp,} against any $\phi\in H^1_0(\Om)$, we deduce
			\begin{align}\label{me}\int_\Om \nabla w_\e\nabla \phi ~dx&=\la \int_\Om \al(x)(h(w_\e)+\e)^{-q}h'(w_\e)  \phi ~dx+ \la\int_\Om f(x,h(w_\e))h'(w_\e) \phi ~dx.
			\end{align}   For $\phi\geq 0,$ by Lemma \ref{L1}-$(h_{13})$, $(h_5)$, $(h_8)$,
			\begin{align}\label{st}\al(x)(h(w_\e)+\e)^{-q}h'(w_\e)\phi\leq \al(x)h(\underline w)^{-q}h'(\underline w)  \phi \leq \|\al\|_\infty \|\phi\|_\infty h(\underline w)^{-q} \leq C h(1)^{-q}\underline w^{-q}.
			\end{align} 
			Again for $\phi\geq 0,$ using the fact that $f(x,h(s))h'(s)$ is increasing in $s>0$, we have 
			\begin{align}\label{aije}f(x,h(w_\e))h'(w_\e)\phi\leq  \|\phi\|_\infty f(x,h(\oline w))h'(\oline w).
			\end{align} Then, \eqref{me}, \eqref{st} and \eqref{aije} combining with the Lebesgue dominated convergence theorem yield that 
			\begin{align}\label{a14}
				\int_\Om \nabla w_\la\nabla \phi ~dx&=\la \int_\Om \al(x)h(w_\la)^{-q}h'(w_\la)  \phi ~dx+\la \int_\Om f(x,h(w_\la))h'(w_\la) \phi ~dx\,\,\,\text{ as $\e\to 0$}.
			\end{align}
			Since for any $\phi\in H^1_0(\Om),$ $\phi=\phi^+-\phi^-,$ the relation in \eqref{a14} holds for all $\phi\in H^1_0(\Om).$ Thus $w_\la$ is a solution to \eqref{pp}. This completes the proof of the lemma.
		\end{proof}
		\noi In the next lemma, we aim to discuss the regularity result for the solutions of \eqref{pp}. 
		\begin{lemma}\label{lem6}
			Assume that $(f1)-(f4)$ and $(\al1)$ hold. Let $h$ be defined as in \eqref{g}.	If $w\in H^1_0(\Om) $ is any  weak solution to \eqref{pp} for $\la \in (0,\Lambda^*]$, then $w \in L^\infty(\Om)\cap C^+_{\varphi_{q}}(\Om)$.
		\end{lemma}
		\begin{proof} Let $w\in H^1_0(\Om)$ be a weak solution to \eqref{pp}.
			First, in spirit of \cite[Lemma A.4]{js}, we show that $w$ is in $L^\infty(\Om)$. For that, let us define  a $C^1$ cut-off  function $\psi:\mb R\to[0,1]$ as \begin{equation*}
				\psi(s)=\left\{
				\begin{array}{l}
					0~~\mbox{if}~~ s\leq 0,\\
					1~~\mbox{if}~~ s\geq 1,
				\end{array}
				\right.
			\end{equation*} with $\psi^\prime(s)\geq0.$  Now for any $\e>0$, define 
			\[\psi_\e(s)=\psi\left(\frac{s-1}{\e}\right) \text{\;\; for } s\in \mb R.\] Note that $\nabla (\psi_\e\circ w)=(\psi_{\e}^{\prime}\circ w)\nabla w$.  Hence, $\psi_\e\circ w\in H^1_0(\Om).$ Let $v\in C^\infty_c(\Om)$ with $v\geq 0$. Now using $\phi:= (\psi_\e\circ w)v$ as a test function in \eqref{wk}, we obtain
			\begin{align}\label{p1}
				\int_{\Om} \nabla w\nabla (\psi_\e\circ w)v ~dx&-\la\int_\Om \al(x)h(w)^{-q} h'(w)(\psi_\e\circ w)v ~dx\notag\\&\qquad\qquad\quad-\la\int_{\Om}f(x,h(w)(x))h'(w)(\psi_\e\circ w)v~dx=0.
			\end{align}	 Since \[ \nabla w\nabla (\psi_\e\circ w)v=|\nabla w|^2 (\psi_\e'\circ w)v+ (\nabla w\nabla v)  (\psi_\e\circ w),\] from \eqref{p1}, we deduce 
			\begin{align*}
				\int_{\Om} (\nabla w\nabla v)  (\psi_\e\circ w) ~dx\leq \la \int_\Om \al(x)h(w)^{-q} h'(w)(\psi_\e\circ w)v ~dx+\la\int_{\Om}f(x,h(w)(x))h'(w)(\psi_\e\circ w)v~dx.
			\end{align*}
			In the last relation, letting $\e\to 0^+$  and using Lemma \ref{L1}-$(h_{13}), (h_3)$, we get
			\begin{align*}
				\int_{\Om} \nabla (w-1)^+\nabla v   ~dx&\leq \la\int_{\Om\cap \{x\,:\,w(x)>1\}} \al(x)h(1)^{-q} h'(1) v ~dx+\la\int_{\Om\cap \{x\,:\,w(x)>1\}}f(x,h(w))h'(w)v~dx\\
				&\leq C +\la \int_{\Om}f(x,h(w))v~dx.
			\end{align*}
			Now following the arguments as in   \cite [Lemma $10$, Theorem $C$]{hirano} combining with Theorem \ref{TM-ineq}, from the last relation, we infer that $(w-1)^+\in L^\infty(\Om).$ Hence $w\in L^\infty(\Om).$\\
			Now we claim that $\uline{w_\la}\leq  w$ a.e. in $\Om.$ Suppose the claim is not true. Then using $(\uline{w_\la}-  w)^+$ as the test function in $-\De (\uline{w_\la}-  w)\leq \la \al(x) (h(\uline{w_\la})^{-q}h^\prime(\uline{w_\la})-  h(w)^{-q}h^\prime(w)) $ in $\Om$ and recalling Lemma \ref{L1}-$(h_{13})$, we deduce
			\begin{align*}
				0\leq\int_\Om| \nabla  (\uline{w_\la}-  w)^+|^2 dx&\leq \la \int_\Om\al(x) (h(\uline{w_\la})^{-q}h^\prime(\uline{w_\la})-  h(w)^{-q}h^\prime(w)) (\uline{w_\la}-  w)^{+} dx\leq0.
			\end{align*} Hence, the claim holds.
			Next, let $z_\la$ be a solution to the  problem, \begin{equation}
				\label{spres}\left\{
				\begin{array}{rlll}
					-\Delta z_\la&=&\la (\al(x) h(z_\la)^{-q} h'(z_\la)+\|g(w)\|_\infty \exp(2 \|w\|_\infty^2)\; \text{in}\;
					\Om,\\
					z_\la&>&0 \; \text{in}\;
					\Om,\\
					z_\la&=&0 ~\mbox{on}~ \partial\Om.
				\end{array}
				\right.
			\end{equation} Then, arguing similarly as above, we get $w\leq z_\la.$ Now, it can be checked that the results in Theorem \ref{sp} hold for the problem \eqref{spres}.  Therefore,  $z_\la$ is unique and $\uline{w_\la}\leq w\leq z_\la$. So, in the light of Theorem \ref{sp}-$(i)$, there exist two positive constants $C_1(\la,q)<<1,$ and $ C_2(\la,q)$ such that
			\begin{align}
				C_1(q,\la)\de (x)&\leq w\leq C_2(q,\la)\de(x) \text {\;\;\; if \;}  0<q<1,\label{Inqq1}\\
				C_1(q,\la)\de (x)^\frac{2}{q+1}&\leq w\leq C_2(q,\la)\de(x)^\frac{2}{q+1} \text {\;\;\; if \;}  1<q<3.\label{Inqq2}
			\end{align} Moreover, if $q=1$,  there exists a constant $C´(\la)>0$ and for any $\e>0$ small enough, there exists a  constant $C_\e(\la)>0$ such that 
			\begin{align}\label{Inqq3}
				C´(\la)\de (x)&\leq w\leq C_\e(\la)\de(x)^{1-\e}.
			\end{align}
			Finally,  combining \eqref{Inqq1}, \eqref{Inqq2}, \eqref{Inqq3} and  recalling  the standard elliptic regularity theory,  it follows that $w\in C_{\varphi_q}^+(\Om).$
		\end{proof}
		\begin{remark}\label{rem1}
			Using  Lemma \ref{lem6}, and adapting the proof of Corollary 1.1 in \cite{adijg}(also see \cite{js}), one can show that if $w\in H^1_0(\Om)$ is  any weak solution to \eqref{pp}, then  $w\in C(\oline \Om)$. Moreover, when $0<q<1$,  $w\in C^1(\oline\Om)$.
		\end{remark}\noi The following lemma  basically ensures the non-existence of solution to \eqref{pp} for  $\la>\La^*$. 
		\begin{lemma}\label{lem2} 
			Let the conditions in Theorem \ref{thm1} hold and let $h$ be defined as in \eqref{g}. Then, $0<\La^*<+\infty.$
		\end{lemma}	
		\begin{proof}
			From Lemma \ref{lem1}, we can infer that $\La^*>0.$ Thus, we are left to show that $\La^*<+\infty.$ Suppose this is not true. Then, there exists a sequence $\{\la_k\}\subset \mc Q$ such that $\la_k\to +\infty$ as $k\to +\infty.$ For $s>0,$ let us define the function	\begin{align*}
				\mc N_{\la}(s):=\la[h(s)^{-q}+f(x,h(s))] \frac{h'(s)}{s}.
			\end{align*}  We claim that there exist $k_0\in \mb N$ sufficiently large and $\beta=\beta (\la_{k_0})>0$ such that for all $s>0$,
			\begin{align}\label{b1}
				\mc N_{\la_{k_0}}(s)=\la_{k_0}[h(s)^{-q}+f(x,h(s))] \frac{h'(s)}{s}>\beta>\tilde\la_{1,\Om}(\varrho),
			\end{align} where $\tilde\la_{1,\Om}(\varrho)$ is the first eigenvalue of the problem  \eqref{ev} with $\varrho(x)=\min\{1,\al(x)\}$.
			Indeed, for any arbitrary $k\in\mb N$ and for $s\in [1/n, n],\; n\in \mb N$, let us consider $\mc N_{\la_k}(s)$.  Since $\mc N_{\la_k}$ is a continuous function, there exists $s_n:=s_{n,k}\in [1/n,n]$ such that
			\[\mc N_{\la_k}(s_n)\leq \mc N_{\la_k}(s) \text{  for all }  s\in [1/n,n].\]  Now we show that, up  to some sub-sequence,  $s_n\to s_{0,k}\in (0,+\infty)$ as $n\to +\infty$.  If not, we have either $s_n\to 0$ or $s_n\to+\infty$ as $n\to+\infty$.  For both the cases, applying Lemma \ref{L1}-$(h_6),(h_8)$, we obtain
			\[\lim_{n\to+\infty}\mc N_{\la_k}(s)\geq \lim_{n\to+\infty}\mc N_{\la_k}(s_n)=+\infty.\] That is, $\mc N_{\la_k}\geq +\infty$ for all $s\in (0,+\infty)$ and for all $k\in \mb N,$ which is absurd. Hence, $s_n\to s_{0,k}\in (0,+\infty)$ as $n\to +\infty$ and 
			\begin{align}\label{b2}
				\mc N_{\la_k}(s)\geq  \la_{k}[h(s_{0,k})^{-q}+f(x,h(s_{0,k}))] \frac{h'(s_{0,k})}{s_{0,k}} \text{  for all } s>0.
			\end{align} Arguing in a similar manner as in \eqref{b2}, it can be deduced that $s_k:=s_{0,k}\to s_0\in (0,+\infty)$, up to some sub-sequence, as $k\to+\infty.$ Using this fact, from \eqref{b2}, we get \eqref{b1}. Thus, the claim follows.\\
			Since $\la_{k_0}\in\mc Q,$
			for $\la=\la_{k_0},$ let  $w_{\la_{k_0}}$ be a solution to \eqref{pp}. So, it follows that
			\begin{align}
				-\De w_{\la_{k_0}}-\beta \varrho(x)w_{\la_{k_0}}&\geq -\De w_{\la_{k_0}}-\mc N_{\la_{k_0}} \varrho(x) w_{\la_{k_0}}\notag\\
				&=-\De w_{\la_{k_0}}-\varrho(x)w_{\la_{k_0}}\la_{k_0}[h(w_{\la_{k_0}})^{-q}+f(x,h(w_{\la_{k_0}}))] \frac{h'(w_{\la_{k_0}})}{w_{\la_{k_0}}}\notag\\
				&\geq -\De w_{\la_{k_0}}-\la_{k_0}[\al(x)h(w_{\la_{k_0}})^{-q}-f(x,h(w_{\la_{k_0}})) {h'(w_{\la_{k_0}})}]=0.\notag
			\end{align}This implies that $ -\De w_{\la_{k_0}}\geq \beta \varrho(x)  w_{\la_{k_0}}>0$ in $\Om$, which in view of strong maximum principle yields that $ w_{\la_{k_0}}>0$ in $\Om$. Now by applying Picone's identity for $\varphi_{1,\Om}$ and $ w_{\la_{k_0}}$, we derive
			\begin{align*}
				0&\leq \int_\Om |\nabla \varphi_{1,\Om}|^2 ~dx-\int_\Om \nabla\left(\frac {\varphi_{1,\Om}^2}{w_{\la_{k_0}}}\right) \nabla  w_{\la_{k_0}}~dx\\
				&\leq \int_\Om |\nabla \varphi_{1,\Om}|^2 ~dx-\int_\Om \beta\varrho(x) \varphi_{1,\Om}^2 ~dx\\
				&=(\tilde\la_{1,\Om}(\varrho)-\beta)\int_\Om \varrho(x) \varphi_{1,\Om}^2 ~dx.
			\end{align*} Therefore, $\tilde\la_{1,\Om}(\varrho)\geq\beta$, which contradicts \eqref{b1}. Thus, the proof of the lemma follows.
		\end{proof}
		\noi In the next result, using  a sub-super solution technique, we  show the existence of at least one solution to \eqref{pp}.
		\begin{proposition}\label{lem3}
			Let the conditions in Theorem \ref{thm1} be satisfied and let $h$ be defined as in \eqref{g}. Then for each $\la\in (0,\La^*)$, \eqref{pp} admits a nontrivial solution in $H^1_0(\Om)\cap C_{\varphi_q}^+(\Om)$.
		\end{proposition} 
		\begin{proof}
			Let $\la \in (0,\La^*)$ and $\la^\prime \in (\la,\La^*)$. Then from the definition of $\La^*$ and Lemma \ref{lem1}, one can see that $w_{\la'}\in H^1_0(\Om)$ forms a weak solution to \eqref{pp} for $\la=\la'$. 
			Let  $\underline {w_\la}$ be  as in Theorem \ref{tsp}. Then
			\begin{align}\label{SS}
				-\De \underline{ w_\la}= \la\al(x)h(\underline {w_\la})^{-q}h'(\underline{ w_\la})\leq \la\al(x)h(\underline {w_\la})^{-q}h'(\underline{ w_\la})+ \la f(x,h(\underline{ w_\la}))h'(\underline {w_\la}),\;\; x\in \Om.
			\end{align}
			Thus, $\underline {w_\la}$ is a weak sub-solution to \eqref{pp}. Therefore, $w_{\la'}$ and $\underline w_\la$ satisfy the following:
			\begin{equation*}
				\left\{
				\begin{array}{l}
					-\De { w_{\la'}} \geq \la\al(x)h({w_{\la'}})^{-q}h'({w_{\la'}}) \; \text{in}\;\;\;
					\Om,\\
					-\De \underline{ w_\la}\leq \la\al(x)h(\underline{ w_\la})^{-q}h'(\underline{ w_\la})\;\;\;	\text{in}\;
					\Om.
				\end{array}
				\right.
			\end{equation*} Hence, Lemma \ref{comp-princ} yields that $\uline{w_\la} \leq w_{\la^\prime}$. Now we consider the closed convex subset $Y_\la$ of $H^1_0(\Om)$ as
			\begin{align}\label{conx}Y_{\la}:= \{w \in H^1_0(\Om):\; \uline{w_\la} \leq w\leq w_{\la^\prime} \}.\end{align}
			Let $\{w_k\} \subset Y_\la$ be such that $w_k \rightharpoonup w_0$ in $H^1_0(\Om)$ as $k \to +\infty$. Then, up to a sub-sequence, $w_k(x)\to w_0(x)$ point-wise a.e. in $\Om$.
			Since $w_{\la^{\prime}}$ is a solution of \eqref{pp}, by Lemma \ref{lem6}, $w_{\la^{\prime}}\in C_{\varphi_q}^+.$ Now for $1<q<3$, using	\eqref{Inqq1}  and Lemma \ref{L1}-$(h_5)$,  we get \begin{align}\label{Dl}\al(x)h(w_k)^{1-q} \leq \al(x) w_k^{1-q} \leq \al(x) {w_{\la'}}^{1-q}\leq \|\al\|_\infty(C_2(\la^{\prime}, q))^{1-q}\de^{1-q}\in L^1(\Om).\end{align}   Next, for $q=1$, by Lemma \ref{L1}-$(h_5)$ and \eqref{Inqq3},  for any sufficiently small $\e>0$ \begin{align}\label{Dl2}\al(x)\log (h(w_k))\leq \al(x)\log (w_k)\leq \al(x)\log ({w_{\la'}})\leq \al(x)w_{\la'}\leq \|\al\|_\infty C_\e(\la^\prime) \de^{1-\e}\in L^{1}(\Om). \end{align}  {For $ 1<q<3$, using $h$  is increasing,   \eqref{inqq2} with $c_1(q,\la)>0$ small enough such that $c_1\de^{\frac{2}{1+q}}<1$ and  Lemma \ref{L1}-$(h_8)$, we obtain
				\begin{align}\label{mon1}\al(x)h(w_k)^{1-q} \leq \al(x)h(\uline {w_\la})^{1-q} \leq \al(x) h(c_1\de^{\frac{2}{1+q}})^{1-q}\leq  \|\al\|_\infty c_1^{1-q} h(1)^{1-q}  \de ^{\frac{2(1-q)}{1+q}}\in L^1(\Om),\end{align} }\noi since $\frac{2(1-q)}{1+q}>-1$. Furthermore, using $(f2)$ in combination with Lemma \ref{L1}-$(h_6)$ and Theorem \ref{TM-ineq}, we deduce \begin{align}\label{Dl3}F(x,h(w_k))<C \exp((1+\e)h(w_k)^4)\leq C \exp(2(1+\e)w_{\la'}^2)\in L^1(\Om).\end{align} Therefore, by the Lebesgue dominated convergence theorem,
			\begin{align*}
				\int_\Om \al(x) h(w_k)^{1-q}~dx &\to \int_\Om \al(x) h(w_0)^{1-q}~dx, \text{  if } q\not =1;\\
				\int_\Om \al(x) \log (h(w_k))~dx &\to \int_\Om \al(x) \log (h(w_0))~dx, \text{  if } q =1;\\
				\int_{\Om} F(x,h(w_k)) ~dx &\to \int_{\Om} F(x,h(w_0)) ~dx.\end{align*}
			Using the last three limits and  the weak lower semicontinuity property of the norm, it follows that $J_\la$ is weakly lower semicontinuous on $Y_\la$. Since $Y_\la$ is weakly sequentially closed subset of $H^1_0(\Om)$, there exists a $w _\la\in Y_\la$ such that
			\begin{equation}\label{gm-3}
				\inf_{w \in Y_\la} J_\la(w) = J_\la(w_\la).
			\end{equation}
			\noi Now we show that $w_\la$ is a weak solution to \eqref{pp}.\\
			For $\varphi \in H^1_0(\Om)$ and $\e>0$ small enough, we define
			\[v_\e := \min\{w_{\la^\prime}, \max\{\uline{w_\la}, w_\la+\e \varphi\} \} = w_\la+\e \varphi - \varphi^\e+ \varphi_\e \in Y_\la,\]
			where $\varphi^\e := \max\{0, w_\la+\e\varphi- w_{\la^\prime}\}$ and $\varphi_\e:=\max\{0, \uline{w_\la}-w_\la-\e\varphi\}$. By construction, $v_\e \in Y_\la$ and $\varphi^\e, \varphi_\e \in H^1_0(\Om)$. Since $w_\la+t(v_\e-w) \in Y_\la$, for each $0<t<1$, using \eqref{gm-3}, Lemma \ref{diff} and mean value theorem, we obtain
			\begin{align*}
				0 &\leq \lim_{t\to 0^+} \frac{J_\la(w_\la+t(v_\e-w_\la))- J_\la(w_\la)}{t} \\
				& = \int_\Om \nabla w_\la \nabla(v_\e-w_\la) ~dx-\la\lim_{t\to 0^+} \int_{\Om}\al(x)h(w_\la+\theta t (v_\e-w_\la))^{-q} h'(w_\la+\theta t (v_\e-w_\la)) (v_\e-w_\la)~dx\\ &\qquad- \la\int_{\Om} f(x,h(w_\la)) h'(w_\la)(v_\e-w_\la)~dx
			\end{align*} for some $0<\theta<1.$ From the definition of $\varphi^\e,\varphi_\e,$ we get that $|v_\e-w_\la|\in H^1_0(\Om)$, which yields that 
			\[|(h(\uline{w_\la})^{-q}h'(\uline{w_\la}))(v_\e-w_\la)|\in L^1(\Om).\] Moreover, using Lemma \ref{L1}-$(h_{13})$, we obtain
			\[|h(w_\la+\theta t (v_\e-w_\la))^{-q} h'(w_\la+\theta t (v_\e-w_\la)) (v_\e-w_\la)|\leq|(h(\uline{w_\la})^{-q}h'(\uline{w_\la}))(v_\e-w_\la)| \] for all $t\in (0,1)$.  From the last relation, it follows that 
			\begin{equation}\label{gm-4}
				\int_{\Om}\nabla w_\la\nabla \varphi~dx {-}\la \int_{\Om}\al(x)h(w_\la)^{-q}h'(w_\la)\varphi ~dx-\la\int_\Om f(x,h(w_\la))h'(w_\la)\varphi~dx \geq \frac{1}{\e} (E^\e-E_\e),
			\end{equation}
			where
			\begin{align*}
				E^\e &:= \int_{\Om}\nabla w_\la\nabla \varphi^\e~dx-\la \int_{\Om}\al(x)h(w_\la)^{-q}h'(w_\la)\varphi^\e ~dx- \la\int_\Om f(x,h(w_\la))h'(w_\la)\varphi^\e~dx;\\
				E_\e &:= \int_{\Om}\nabla w_\la\nabla \varphi_\e~dx-\la \int_{\Om}\al(x)h(w_\la)^{-q}h'(w_\la)\varphi_\e ~dx- \la\int_\Om f(x,h(w_\la))h'(w_\la)\varphi_\e~dx.\\
			\end{align*}
			We define the set $\Om^\e := \{x \in \Om :\; (w_\la+\e\varphi)(x) \geq w_{\la^\prime}(x)> w_\la(x)\}$ so that $\mc |\Om^\e| \to 0$ as $\e \to 0^+$.
			Next, using the fact that ${w_{\la^\prime}}$ is a super-solution to \eqref{pp} together with Lemma \ref{L1}-$(h_{13})$, we estimate the following:
			\begin{align*}
				\frac{1}{\e}E^\e=	&\frac{1}{\e}\left[ \int_{\Om}\nabla (w_\la-w_{\la'})\nabla\varphi^\e ~dx +  \int_{\Om}\nabla w_{\la'}\nabla\varphi^\e ~dx - \la \int_{\Om}(\al(x)h(w_\la)^{-q}+ f(x,h(w_\la)))h'(w_\la)\varphi^\e~dx\right] \\
				&\geq \frac 1\e\int_{\Om^\e}|\nabla (w_\la-w_{\la'})|^2 ~dx + \int_{\Om^\e}\nabla (w_\la-w_{\la'})\nabla\varphi ~dx+  \frac{\la}{\e}\int_{\Om^\e}\al(x) (h(w_{\la^\prime})^{-q}h'(w_{\la'})\\
				&\quad -h(w_\la)^{-q}h'(w_\la))\varphi^\e ~dx+\frac {\la}{\e}\int_{\Om^\e}(f(x,h(w_{\la'}))h'(w_{\la'})-f(x,h(w_{\la}))h'(w_{\la}))\varphi^\e ~dx\\
				&  \geq \int_{\Om^\e}\nabla (w_\la-w_{\la'})\nabla\varphi ~dx-{\la}\int_{\Om^\e}\al(x) (h(w_{\la^\prime})^{-q}h'(w_{\la'})-h(w_\la)^{-q}h'(w_\la))|\varphi|~dx\\
				&\qquad\quad-\la\int_{\Om^\e}|f(x,h(w_{\la'}))h'(w_{\la'})-f(x,h(w_{\la}))h'(w_{\la})|\,|\varphi| ~dx\\& =o(1)\; \text{as}\; \e \to 0^+.
			\end{align*}
			Arguing similarly, we have
			\[\frac{1}{\e}E_\e \leq o(1) \; \text{as}\; \e \to 0^+.\]
			Thus, from \eqref{gm-4}, we get 
			\[ \int_{\Om}\nabla w_\la\nabla \varphi~dx {-}\la \int_{\Om}\al(x)h(w_\la)^{-q}h'(w_\la)\varphi ~dx-\la\int_\Om f(x,h(w_\la))h'(w_\la)\varphi~dx  \geq o(1)\; \text{as}\; \e \to 0^+\] for all $\varphi \in {H_0^1(\Om)}$.
			Considering $-\varphi$ in place of $\varphi$ and  following the similar arguments  as above, we infer  that $w_\la$ is a weak solution to \eqref{pp}. Moreover, from the construction of $w_\la$ and Lemma \ref{lem6}, it follows that $w_\la\in C_{\varphi_q}^+$. This concludes the proof of the proposition. 
		\end{proof}  
		\begin{lemma}\label{lem4}
			Assume that the conditions in Theorem \ref{thm1} hold and let $h$ be defined as in \eqref{g}. Let $\la\in (0,\La^*)$. Then any weak solution to \eqref{pp} obtained in Proposition \ref{lem3} is a local minimizer  for the functional $J_\la.$
		\end{lemma} 
		\begin{proof} We prove this lemma for the case $q\not=1$. For $q=1$, the proof follows in a similar fashion.\\
			Now suppose the statement of the lemma does not hold. So,  let us assume that $w_\la$ is not a local minimum of $J_\la$, where $w_\la$ is a solution to \eqref{pp} obtained in Lemma \ref{lem3}. Then there exists a sequence $\{w_k\}\subset H_0^1(\Om)$ such that
			\begin{equation}\label{gm-5}
				\|w_k-w_\la\| \to 0\; \text{as}\; k \to +\infty \; \text{and}\; J_\la({w_k})< J_\la(w_\la).
			\end{equation}
			Next, we define $\uline w:= \uline{w_\la}$ and $\oline{w}:= w_{\la^\prime}$ as a sub-solution  and a super-solution to \eqref{pp}, respectively, as defined in the proof of Proposition \ref{lem3}. Furthermore,  we define
			\begin{equation*}
				v_k := \max\{\uline{w}, \min\{w_k,\overline{w}\}\}=\left\{
				\begin{array}{ll}
					\uline{w},\; &\text{if}\; w_k<\uline{w},\\
					w_k,\; &\text{if}\; \uline{w}\leq w_k\leq\oline{w},\\
					\oline{w},\; &\text{if}\; w_k>\overline{w},
				\end{array}
				\right.\\
			\end{equation*}
			\begin{align*} &\uline{u_k}:= (w_k-\uline{w})^-,\;\; \oline{u_k}:= (w_k-\oline{w})^+,\\  &\uline{\mc S_k}:= \text{supp}(\uline{u_k}),\;\;\; \oline{\mc S_k}:= \text{supp}(\oline{u_k}).
			\end{align*} Then, $w_k = v_k - \uline{u_k}+\oline{u_k}$ and $v_k \in Y_\la,$ where the set $Y_\la$ is defined in \eqref{conx}. 
			Then we can express $J_\la(w_k)$ as \begin{align}\label{x}J_\la(w_k)= J_{\la}(v_k)+ A_k+B_k,
			\end{align}where
			\begin{align}\label{gm-6}
				A_k &:=\frac 12 \int_{\oline{\mc S_k}}(|\nabla w_k|^2-|\nabla \oline{w}|^2)~dx-	 \frac{\la}{1-q}\int_{\oline{\mc S_k}}\al(x) (h(w_k)^{1-q}-h(\oline w)^{1-q} )~dx\notag\\&\qquad\qquad\quad -\la\int_{\oline{\mc S_k}} (F(x,h(w_k)) - F(x,h(\oline{w})))~dx,\\
				B_k &:=\frac 12 \int_{\uline{\mc S_k}}(|\nabla w_k|^2-|\nabla \uline{w}|^2)~dx-	 \frac{\la}{1-q}\int_{\uline{\mc S_k}}\al(x) (h(w_k)^{1-q}-h(\uline w)^{1-q} )~dx\notag\\&\qquad\qquad\quad -\la\int_{\uline{\mc S_k}} (F(x,h(w_k)) - F(x,h(\uline{w})))~dx.\label{gm7}
			\end{align}From Proposition \ref{lem3}, we get $J_\la(w_k)\geq J_{\la}(w_\la)+ A_k+B_k.$ We intend to show that $A_k,B_k\geq 0$ for large $k$, as we will see later \begin{align}\label{sk}\lim\limits_{k \to +\infty}|\oline{\mc S_k}|=0\;\; \text{\; and \;} \lim\limits_{k \to +\infty}|\uline{\mc S_k}|=0.
			\end{align}
			Let us prove \eqref{sk}. 	For any $b>0$, let us define the set $$\Om_b:=\{x\in \, :\, \de (x)>b\}.$$
			Now   recalling  the proof of Proposition \ref{lem3}, we have \[\oline w\geq w_\la>\uline w:=\uline{w_\la}>0.\] On the other hand, by $(f1)$ and Lemma \ref{L1}, we have $f(x,h(s))h'(s)$ is increasing in $s>0$.
			Therefore, combining the   above facts together with Lemma \ref{L1}-$(h_3), (h_{12})$ and using the  mean value theorem and \eqref{inqq1}, \eqref{inqq2}, for $x\in \Om_{\frac b2}$,we get
			\begin{align}
				-\De(\oline w-w_\la)\geq &\la\al(x)[h(\oline w)^{-q}h'(\oline w)-h(w_\la)^{-q}h'(w_\la)]\notag\\
				&=\la \al(x)[-q(h'(w_*))^2h(w_*)^{-q-1}+h(w_*)^{-q}h''(w_*)](\oline w-w_\la),\;\;\;\; \text{ for some } w_*\in (w_\la,\oline w)\notag\\
				&\geq \la \al(x) [-q h(w_*)^{-q-1}-\sqrt{2}h(w_*)^{-q}](\oline w-w_\la)\notag\\
				&\geq \la \al(x) [-q h(\uline {w_\la})^{-q-1}-\sqrt{2}h(\uline {w_\la})^{-q}](\oline w-w_\la)\notag\\
				&\geq -\la \al(x)\left [q \,h\left (c_1(q,\la)\left(\de(x)\right)^{a}\right )^{-q-1}+\sqrt{2}h\left(c_1(q,\la)\left(\de(x)\right)^{a}\right)^{-q}\right](\oline w-w_\la)\notag\\
				&\geq -\la \al(x)\left [q h\left (c_1(q,\la)\left(\frac{b}{2}\right)^{a}\right )^{-q-1}+\sqrt{2}h\left(c_1(q,\la)\left(\frac{b}{2}\right)^{a}\right)^{-q}\right](\oline w-w_\la),\notag
			\end{align} where $a=1$, if $0<q<1$ and $a=\frac {2}{q+1}$, if $1<q<3.$
			Applying \cite[Theorem 3] {bn}, from the last relation, we infer that for any $b>0$, there exists a constant $C_*>0$ such that $$\oline w-w_\la\geq C_* \frac b 2>0\;\;\; \text{in } \Om_b.$$ Given $\e>0$, choose $b>0$ such that $|\Om\setminus \Om_b|<\frac \e2$. Since we assumed that $w_k\to w_\la$ in $H^1_0(\Om)$, for sufficiently large $k\in\mb N$, we obtain
			\begin{align*}
				|\oline{\mc S_k}|&\leq |\Om\setminus\Om_b|+|\oline{\mc S_k}\cap \Om_b|\\
				&\leq \frac \e 2 +\int_{\oline{\mc S_k}\cap \Om_b} \frac{w_k-w_\la}{\oline w-w_\la}~dx\\
				&\leq \frac \e 2 +\frac{4}{C_*^2 b^2}\int_{\oline{\mc S_k}\cap \Om_b} {(w_k-w_\la)^2}~dx\\
				&<\e + C \|w_k-w_\la\|^2.
			\end{align*}
			This yields that $|\oline{\mc S_k}|\to 0$ as $k\to+\infty.$ In a similar fashion as above, considering $\uline w-w_\la$, we get  $|\uline{\mc S_k}|\to 0$ as $k\to+\infty.$ Therefore, as $k\to+\infty$
			\begin{align}\label{c2}
				\|\oline{v_k}\|^2&=\int_{\oline{\mc S_k}}|\nabla(w_k-\oline w)|^2 ~dx\notag\\
				&\leq 2\left(\|w_k-w_\la\|^2+ \int_{\oline{\mc S_k}}|\nabla(w_\la-\oline w)|^2 ~dx\right)\to 0.
			\end{align}
			\noi Similarly, $\|\uline{v_k}\|^2\to 0$ as $k\to+\infty$. 
			Since $\oline w$ is a super-solution to \eqref{pp}, using Lemma \ref{L1}-$(h_9)$  and mean value theorem, from \eqref{gm-6}, we obtain
			\begin{align}\label{111}
				A_k&=\frac 12 \int_{\oline{\mc S_k}}(|\nabla( \oline w+ \oline{v_k})|^2-|\nabla \oline{w}|^2)~dx-	 \frac{\la}{1-q}\int_{\oline{\mc S_k}}\al(x) (h(\oline w+ \oline{v_k})^{1-q}-h(\oline w)^{1-q} )~dx\notag\\&\qquad\qquad -\la\int_{\oline{\mc S_k}} (F(x,h(\oline w+\oline{v_k})) - F(x,h(\oline{w})))~dx,\notag\\
				&=\frac 12 \|\oline{v_k}\|^2+\int_{\oline{\mc S_k}} \nabla \oline w\nabla \oline{v_k} ~dx- \frac{\la}{1-q}\int_{\oline{\mc S_k}}\al(x) h(\oline w+\theta  \oline{v_k})^{-q}h'(\oline w+\theta  \oline{v_k}) \oline{v_k} ~dx\notag\\&\qquad\qquad-\la\int_{\oline{\mc S_k}} f(x,h(\oline w+\theta  \oline{v_k})) h'(\oline w+\theta  \oline{v_k}) \oline{v_k}~dx,\;\;\; \theta \in(0,1)\notag\\
				&\geq \frac 12 \|\oline{v_k}\|^2+ \la \int_{\oline{\mc S_k}}\al(x)h(\oline w)^{-q}h'(\oline w)\oline{v_k}~dx
				+\la\int_{\oline{\mc S_k}}f(x,h(\overline{ w}))h'(\overline{w})\oline{v_k}~dx\notag\\
				&\qquad\qquad- \frac{\la}{1-q}\int_{\oline{\mc S_k}}\al(x) h(\oline w+\theta \la  \oline{v_k})^{-q}h'(\oline w+\theta  \oline{v_k}) \oline{v_k} ~dx-\la\int_{\oline{\mc S_k}} f(x,h(\oline w+\theta  \oline{v_k})) h'(\oline w+\theta  \oline{v_k}) \oline{v_k}~dx\notag\\
				&\geq \frac 12 \|\oline{v_k}\|^2
				+\la\int_{\oline{\mc S_k}}f(x,h(\overline{ w}))h'(\overline{w})\oline{v_k}~dx
				-\la\int_{\oline{\mc S_k}} f(x,h(\oline w+\theta  \oline{v_k})) h'(\oline w+\theta  \oline{v_k}) \oline{v_k}~dx\notag\\
				&= \frac 12 \|\oline{v_k}\|^2
				+\la \theta\int_{\oline{\mc S_k}}\left( f'(x,h(\oline w+\tilde \theta  \oline{v_k}))( h'(\oline w+\tilde \theta  \oline{v_k}))^2 +f(x,h(\oline w+\tilde \theta  \oline{v_k})) h''(\oline w+\tilde \theta  \oline{v_k}) \right)\oline{v_k}^2 ~dx,\;\;\; \tilde\theta\in(0,1)\notag\\
				&= \frac 12 \|\oline{v_k}\|^2
				+ \la\theta\int_{\oline{\mc S_k}}\Big( f'(x,h(\oline w+\tilde \theta  \oline{v_k}))( h'(\oline w+\tilde \theta  \oline{v_k}))^2 \notag\\&\qquad\qquad\qquad\qquad\qquad-2f(x,h(\oline w+\tilde \theta  \oline{v_k})) h(\oline w+\tilde \theta  \oline{v_k})( h'(\oline w+\tilde \theta  \oline{v_k}))^4 \Big)\oline{v_k}^2 ~dx.
			\end{align}
			Now by the definition of the function $f$, we have $$f'(x,h(s))=(g'(x,h(s)) +4h(s)^3 g(x,h(s)))\exp(h(s)^4)\geq 4h(s)^3 f(x,h(s)).$$ Using this combining with Lemma \ref{L1}-$(h_{13}), (h_{12}), (h_3), (h_6)$ and H\"older's inequality, from \eqref{111}, we deduce
			\begin{align}
				A_k&\geq \frac 12 \|\oline{v_k}\|^2
				+ \la\theta\int_{\oline{\mc S_k}}\Big( 4 f(x,h(\oline w+\tilde \theta  \oline{v_k}))h(\oline w+\tilde \theta  \oline{v_k})^3( h'(\oline w+\tilde \theta  \oline{v_k}))^2 \notag\\&\qquad\qquad\qquad\qquad-2f(x,h(\oline w+\tilde \theta  \oline{v_k})) h(\oline w+\tilde \theta  \oline{v_k})( h'(\oline w+\tilde \theta  \oline{v_k}))^4 \Big)\oline{v_k}^2 ~dx\notag\\
				&\geq \frac 12 \|\oline{v_k}\|^2
				- \la\theta\int_{\oline{\mc S_k}}\Big( 4 f(x,h(\oline w+\tilde \theta  \oline{v_k}))h(\oline w+\tilde \theta  \oline{v_k})^3( h'(\oline w+\tilde \theta  \oline{v_k}))^3\notag \\&\qquad\qquad\qquad\qquad-2f(x,h(\oline w+\tilde \theta  \oline{v_k})) h(\oline w+\tilde \theta  \oline{v_k}) h'(\oline w+\tilde \theta  \oline{v_k}) \Big)\oline{v_k}^2 ~dx\notag\\
				&\geq \frac 12 \|\oline{v_k}\|^2
				-\la\sqrt{2} \theta\int_{\oline{\mc S_k}} f(x,h(\oline w+\tilde \theta  \oline{v_k}))  \oline{v_k}^2~dx\notag\\
				&\geq \frac 12 \|\oline{v_k}\|^2
				-\la C\int_{\oline{\mc S_k}} \exp((1+\e)h(\oline w+\tilde \theta  \oline{v_k})^4)  \oline{v_k}^2~dx, \;\;\;\;\;\;\;\; \text{ by  } (f2) \text{\; for } \e>0\notag\\
				&\geq \frac 12 \|\oline{v_k}\|^2
				-\la C\int_{\oline{\mc S_k}} \exp(2(1+\e)(\oline w+\tilde \theta  \oline{v_k})^2)  \oline{v_k}^2~dx\notag\\
				&\geq \frac 12 \|\oline{v_k}\|^2
				-\la C\left(\int_{\oline{\mc S_k}} \exp(4(1+\e)(\oline w+\tilde \theta  \oline{v_k})^2) ~dx\right)^{\frac 12} \left(\int_{\oline{\mc S_k}}\oline{v_k}^4~dx\right)^{\frac 12}\notag\\
				&\geq \frac 12 \|\oline{v_k}\|^2
				-\la C |\oline {\mc S_k}| \|\oline{v_k}\|^2\geq 0, \;\;\; \text{  for large }  k\in\mb N,\notag
			\end{align} where in the last line, we used Theorem \ref{TM-ineq}, \eqref{c2} and \eqref{sk}. Similarly, from \eqref{gm7}, we can show that $B_k\geq 0$ for sufficiently large $k\in\mb N$.
			Thus from \eqref{x},  for large  $k\in\mb N$, it yields that 
			\[J_\la(w_k) {\geq}  J_{\la}(w_\la),\]
			which is a contradiction to \eqref{gm-5}. Hence, $w_\la$ is a local minimum of $J_\la$ over $H^1_0(\Om)$. 
		\end{proof}
		
		\noi The next result is a consequence of Lemma \ref{lem4}, which yields that at the threshold level of $\la$, that is for  $\la=\La^*$, we have a weak solution to \eqref{pp}.

		\begin{proposition} \label{lem5}
			Let the conditions in Theorem \ref{thm1} be satisfied and let $h$ be defined as in \eqref{g}. Then for $\la=\La^*$,  \eqref{pp} admits a weak solution in $H^1_0(\Om)\cap C_{\varphi_q}^+(\Om).$
		\end{proposition}
		
		\begin{proof}
			From the definition of $\La^*$, there exists a increasing sequence $\{\la_k\}\in \mc Q$  such that  $\la_k\uparrow\La^*$  as $k \to +\infty$. Hence, by Proposition \ref{lem3}, $\{w_{\la_k}\} \in H^1_0(\Om)\cap C_{\varphi_q}^+(\Om)\cap Y_{\la_k}$ is a sequence of positive weak solutions to \eqref{pp} with $\la=\la_k$, where $Y_{\la_k}$ is defined as in \eqref{conx}. Therefore,
			\begin{align}\label{d1}
				\int_{\Om}|\nabla w_{\la_k}|^2 ~dx-\la_k \int_{\Om}\al(x)h(w_{\la_k})^{-q}h'(w_{\la_k}) w_{\la_k} ~dx-\la_k\int_\Om f(x,h(w_{\la_k}))h'(w_{\la_k})w_{\la_k}~~dx=0.
			\end{align}
			For $\la=\la_k$, let $\uline{w_{\la_k}}$ denote the unique solution to \eqref{sp}, which is a  sub-solution to \eqref{pp}  as in \eqref{SS}. So, using Lemma \ref{L1}-$(h_4)$, we obtain
			\begin{align}\label{d2}
				\int_{\Om}|\nabla \uline{w_{\la_k}}|^2 ~dx= \la_k \int_{\Om}\al(x)h(\uline{w_{\la_k}})^{-q}h'(\uline{w_{\la_k}}) \uline{w_{\la_k}} ~dx\leq  \la_k \int_{\Om}\al(x)h(\uline{w_{\la_k}})^{1-q} ~dx.
			\end{align} 
			Now Lemma \ref{lem4} yields that
			$w_{\la_k}$ is a local minimizer of $J_{\la_k}$ for each $k\in\mb N$.
			So,
			\begin{equation}\label{d3}
				J_{\la_k}(w_{\la_k})=\min_{w\in Y_{\la_k}} J_{\la_k}(w)\leq 	J_{\la_k}(\uline{w_{\la_k}})\leq\left\{
				\begin{array}{l}
					\frac{1}{2}\displaystyle\int_{\Om} |\nabla \uline{w_{\la_k}}|^{2}~dx-\frac{\la_k}{1-q}\int_\Om \al (x)|h(\uline{w_{\la_k}})|^{1-q}~dx~~\mbox{\;\;\; if}~~ q\not =1;\\
					\frac{1}{2}\displaystyle\int_{\Om} |\nabla \uline{w_{\la_k}}|^{2}~dx-\la_k\int_\Om \al (x)\log|h(\uline{w_{\la_k}})|~dx~~\mbox{\;\;\; if}~~ q =1.
				\end{array}
				\right.
			\end{equation} Plugging  \eqref{d2} in \eqref{d3}, we get 
			\begin{equation}\label{d4}
				J_{\la_k}(w_{\la_k})\leq \beta_k:=\left\{
				\begin{array}{l}
					\la_k\left(\frac{1}{2}-\frac{1}{1-q}\right)\ds\int_\Om \al (x)|h(\uline{w_{\la_k}})|^{1-q}~dx~~\mbox{\;\;\; if}~~ q\not =1;\\
					\frac {\la_k}{2}\ds\int_\Om\al(x)~dx-\la_k\ds\int_\Om \al (x)\log|h(\uline{w_{\la_k}})|~dx~~\mbox{\;\;\; if}~~ q =1.
				\end{array}
				\right.
			\end{equation}
					Thus, for all $0<q<3$, since $0<\la_1\leq \la_2\leq\cdots\leq\la_k\leq\cdots\leq \La^*,$ 
					from \eqref{inqq1} and \eqref{inqq2} and Theorem \ref{tsp}-$(iv)$,
					we infer that \begin{align}\label{d5}\sup_{k}\beta_k<+\infty.
					\end{align} {\bf Case-I:} $0<q<1$.\\
					Now \eqref{d1} and \eqref{d4}  combining with Lemma \ref{L1}-$(h_4)$ and $(f3)$   imply that
					\begin{align}\label{d6}
						\beta_k+ \la_k\left(\frac{1}{1-q}-\frac{1}{2}\right)\ds\int_\Om \al (x)h({w_{\la_k}})^{1-q}~dx&\geq	\frac {\la_k}{2} \left(\int_\Om f(x,h(w_{\la_k}))h'(w_{\la_k})w_{\la_k} ~dx-\int_\Om F(x,h(w_{\la_k}))~dx\right)\notag\\
						&\geq  \la_k\left(	\frac 14-\frac 1\tau\right) \int_\Om f(x,h(w_{\la_k}))h(w_{\la_k}) ~dx.
					\end{align}
					Plugging \eqref{d6} and \eqref{d5} in \eqref{d1} and using Lemma \ref{L1}-$(h_4), (h_5)$  together with the fact that $\tau>4$, the Sobolev inequality and the Hölder inequality,  we obtain
					\begin{align}\label{d7}
						\|w_{\la_k}\|^2 
						\leq C_1+C_2\La^* \int_{\Om}\al(x)h(w_{\la_k})^{1-q}~dx \leq C_1+C_2\La^*\|\al\|_\infty \int_{\Om}w_{\la_k}^{1-q}~dx\leq C_1+C_3 \|w_{\la_k}\|^{1-q}.
					\end{align} 
					{{\bf Case-II:} $q=1.$\\
						Again using \eqref{d1}, \eqref{d4},  Lemma \ref{L1}-$(h_4)$ and $(f3)$  combining with the inequality $\log h(s)\leq\log s< s$, for $s>0$, we get 
						\begin{align}\label{logg}\beta_k&+ \la_k\ds\int_\Om \al (x) dx\notag\\&\geq \frac {\la_k}{2}\int_\Om\al(x) h(w_{\la_k})^{-1}h'(w_{\la_k})w_{\la_k}  dx+\frac {\la_k}{2} \left(\int_\Om f(x,h(w_{\la_k}))h'(w_{\la_k})w_{\la_k} ~dx-\int_\Om F(x,h(w_{\la_k}))~dx\right)\notag\\
							& \geq \frac{ \la_k}{ 4}\int_\Om \al(x) dx+\la_k\left(	\frac 14-\frac 1\tau\right) \int_\Om f(x,h(w_{\la_k}))h(w_{\la_k}) ~dx.
						\end{align} This gives that 
						\begin{align}\label{loggg}
							\beta_k+\frac 34\la_k\ds\int_\Om \al (x)\geq \la_k \left(	\frac 14-\frac 1\tau\right) \int_\Om f(x,h(w_{\la_k}))h(w_{\la_k}) ~dx.
						\end{align} Using \eqref{loggg} and \eqref{d5}, from \eqref{d1}, we deduce
						\begin{align}\label{d08}\|w_{\la_k}\|^2\leq C_4( \beta_k+\La^*\|\al\|_\infty|\Om|)<+\infty.
					\end{align}}
					{\bf Case-III:} $1<q<3$.\\
					From \eqref{d6}, it follows that \[ \la_k\left(	\frac 14-\frac 1\tau\right) \int_\Om f(x,h(w_{\la_k}))h(w_{\la_k}) ~dx\leq \beta_k+\la_k \left(\frac{1}{q-1}+\frac{1}{2}\right)\ds\int_\Om \al (x)h({w_{\la_k}})^{1-q}~dx.\] Employing the last relation in \eqref{d1} and using Theorem \ref{tsp}-$(iv)$, Lemma \ref{L1}-$(h_8)$ and \eqref{inqq2} for  $\la_2$ with $c_1(q,\la_2)>0$ small enough such that $c_1(q,\la_2)\de^{\frac{2}{1+q}}<1$ , we get
					\begin{align}\label{d8}
						\|w_{\la_k}\|^2&\leq C_5+C_6\la_k \int_{\Om}h(w_{\la_k})^{1-q}~dx\notag\\
						&\leq C_5+C_6\la_k \int_{\Om}h(\uline{ w_{\la_k}})^{1-q}~dx\notag\\
						&\leq C_5+C_6\La^* \int_{\Om}h(\uline{w_{\la_2}})^{1-q}~dx\notag\\
						&\leq C_5+C_6\La^* \int_{\Om}h\left(c_1(q,\la_2)\de^{\frac{2}{1+q}}\right)^{1-q}~dx\notag\\
						& \leq C_5+C_6  \La^* h(1)^{1-q}(c_1(q,\la_2))^{1-q}\int_{\Om} \de^{\frac{2(1-q)}{1+q}}~dx<+\infty,
					\end{align} since $\frac {2(1-q)}{1+q}>-1.$\\\\
					Thus,  each of the expressions in \eqref{d7}, \eqref{d8}, \eqref{d08} from Case-I, II, III, respectively, yields that \begin{align}\label{mon2}\ds\limsup_{k\to+\infty}\|w_{\la_k} \|<+\infty.\end{align}  Therefore, up to a sub-sequence, there exists $w_{\La^*}\in H^1_0(\Om)$ such that $w_{\la_k}\rightharpoonup w_{\La^*}$ weakly  in $H^1_0(\Om)$  and $w_{\la_k}(x)\to w_{\La^*}(x)$ a.e. in $\Om$ as $k \rightarrow +\infty$. Also, by the construction, it follows that $w_{\la_k} \geq\uline{w_{\la_k}} \geq\underline{w_{\la_1}}$. So,\[{w_{\La^*}}(x)=\lim_{k\to +\infty} {w_{\la_k}}(x)>\uline{w_{\la_1}}(x)>0 \text{  a.e. in } \Om.\] Thus, by  the Lebesgue dominated convergence theorem, for any $\phi\in C_c^\infty(\Om)$, we have \begin{align}\label{d9}\int_\Om \al(x) h(w_{\la_k})^{-q}h'(w_{\la_k})\phi ~dx\to\int_\Om \al(x) h(w_{\La^*})^{-q}h'(w_{\La^*})\phi ~dx \text{\;\; as } k\to +\infty.
					\end{align} Next, we show that for any $\phi\in C_c^\infty(\Om),$ 
					\begin{align}\label{ds1}\int_\Om f(x,h(w_{\la_k}))h'(w_{\la_k})\phi~dx\to \int_\Om f(x,h(w_{\La^*}))h'(w_{\La^*})\phi~dx \text{  as  } k\to +\infty.\end{align}  To prove \eqref{ds1}, first observe that for  $1<q<3$, using Lemma \ref{L1}-$(h_4)$, and arguing similarly as in \eqref{mon1}, $\al(x) h(w_{\la_k})^{-q}h'(w_{\la_k})w_{\la_k} \leq\|\al\|_\infty h(w_{\la_k})^{1-q}\leq \|\al\|_\infty h(\uline {w_{\la_1}})^{1-q}\in L^1(\Om)$.  Thus, by  the Lebesgue dominated convergence theorem, we have \begin{align}\label{d9}\int_\Om \al(x) h(w_{\la_k})^{-q}h'(w_{\la_k})w_{\la_k} ~dx\to\int_\Om \al(x) h(w_{\La^*})^{-q}h'(w_{\La^*})w_{\La^*} ~dx \text{\;\; as } k\to +\infty.
					\end{align}  Furthermore, for $0<q\leq1$, \eqref{d9} follows similarly as in \eqref{a5}.
					Hence, from \eqref{d1}, \eqref{mon2} and \eqref{d9}, we obtain \[\limsup_{k\to+\infty}\int_\Om f(x,h(w_{\la_k}))h'(w_{\la_k})w_{{\la_k}}~dx<+\infty.\]
					Then repeating a similar argument as in \eqref{FoF}, we get \eqref{ds1}. 
				Gathering \eqref{ds1} and \eqref{d9},  we  finally infer that $w_{\La^*}\in H^1_0(\Om)$ is a positive weak solution to \eqref{pp}. Moreover, in light of  Lemma \ref{lem6}, we have $w_{\La^*}\in H^1_0(\Om)\cap C_{\varphi_q}^+(\Om)$. Hence, the proof of the proposition is complete.
			\end{proof}

			\noi {\bf Proof of Theorem \ref{thm1} :}  
			Combining Lemma \ref{lem1}, \ref{lem6}, \ref{lem2}, \ref{lem4} along with Proposition \ref{lem3}, \ref{lem5}, we infer that $w_\la\in H_0^1(\Om)\cap C^+_{\varphi_q}(\Omega)$ is a weak solution to \eqref{pp}. Now by Lemma \ref{L1}-$(h_1)$, we have $h$ is a $C^\infty$  function and Lemma \ref{L1}-$(h_8), (h_{11})$ ensure that $h(s)$ behaves like $s$ when  $s$ is close to $0$. Therefore, we can conclude that $h(w_\la) \in H_0^1(\Om)\cap C^+_{\varphi_q}(\Omega)$ forms a weak solution to the problem \eqref{pq}.
			\section{Proof of Theorem \ref{thm2} : Multiplicity result}   
			\noi	This section is dedicated toward establishing  the existence of second solution of \eqref{pp} using the  mountain pass lemma in combination with   Ekeland variational principle.  Let us define the set
			\[T := \{w\in H^1_0(\Om):\; w\geq w_\la\; \text{a.e. in}\; \Om \}.\]
			Since by Lemma \ref{lem4}, $w_\la$ is a local minimizer for $J_\la,$ it follows that $J_\la(w) \geq  J_\la(w_\la)$ whenever $\|w_\la-w\|\leq \sigma_0$, for some small constant  $\sigma_0>0$. Then, one of the following cases holds:
			\begin{enumerate}
				\item[$(ZA)$]{\bf (Zero Altitude):} $\inf \{ J_\la(w)\; |\; w \in T, \;\|w-w_\la\|=\sigma\}=  J_\la(w_\la)$ for all $\sigma\in (0,\sigma_0)$.
				\item[$(MP)$]{\bf (Mountain Pass):} There exists a $\sigma_1 \in (0,\sigma_0)$ such that $\inf\{ J_\la(w)|\; w \in T,\; \|w-w_\la\|=\sigma_1\}> J_\la(w_\la)$.
			\end{enumerate}
			Now for the case $(ZA)$, inspired by \cite{haito} and \cite{adijg}, we prove the existence of second weak solution to \eqref{pp} in the following result.
			\begin{proposition}\label{prp1}
				Let the conditions in Theorem \ref{thm2} hold and let $\la\in(0,\La^*)$. Suppose that $(ZA)$ holds. Then for all $\sigma \in (0,\sigma_0)$, \eqref{pp} admits a  second solution $v_\la \in H^1_0(\Om) \cap C_{\varphi_q}^+(\Om)$  such that  $v_\la\geq w_\la$ and $\|v_\la-w_\la\| = \sigma$.
			\end{proposition} 
			\begin{proof}
				Let us fix $\sigma \in (0, \sigma_0)$ and $r>0$ such that $\sigma -r>0$ and $\sigma +r< \sigma_0$. Define the set
				\[\mc A : = \{w \in T \,:\,\; 0<\sigma-r \leq \|w-w_\la\|\leq \sigma+r\}.\]
				Clearly $\mc A$ is closed in $H^1_0(\Om)$ and by $(ZA)$, $\ds\inf_{w \in \mc A} J_\la(w)=  J_\la(w_\la)$. So, for any minimizing sequence $\{w_k\}\subset \mc A$ satisfying $\|w_k-w_\la\|= \sigma$, by Ekeland variational principle,   we get another sequence $\{v_k\}\subset \mc A$ such that
				\begin{equation}\label{sec-sol-gm1}
					\left\{
					\begin{array}{ll}
						J_\la(v_k) &\leq  J_\la(w_k)\leq J_\la(w_\la)+ \frac{1}{k}\\
						\|w_k-v_k\| &\leq \frac{1}{k}\\
						J_\la(v_k)& \leq J_\la(v)+\frac{1}{k}\|v-v_k\|\;\; \text{for all}\; v \in \mc A.
					\end{array}\right.
				\end{equation}
				For $z \in T,$	we can choose $\e >0$ small enough so that $v_k + \e(z-v_k) \in \mc A$. So, from \eqref{sec-sol-gm1}, we obtain
				\[\frac{ J_\la(v_k + \e(z-v_k)) - J_\la(v_k)}{\e} \geq -\frac{1}{k}\|z-v_k\|.\]
				Letting $\e \to 0^+$, using the fact that $v_k\geq w_\la$ for each $k\in \mb N$ and following the similar arguments as in \eqref{1} and \eqref{kk} for the singular term in the last relation, we obtain
				\begin{equation}\label{sec-sol-gm2}
					-\frac{1}{k}\|z-v_k\|\leq	\int_{\Om}\nabla v_k\nabla(z-v_k) ~dx-\la\int_\Om\al(x) h(v_k)^{-q}h'(v_k)(z-v_k)~dx -\la\int_\Om f(x,h(v_k))h'(v_k)(z-v_k)~dx
				\end{equation} for all $z\in T$.
				Since $\{v_k\}$ is a bounded sequence in $H^1_0(\Om)$,  there exists  $v_\la \in H_0^1(\Om)$ such that, up to a sub-sequence, $v_k \rightharpoonup v_\la$ weakly in $H^1_0(\Om)$ and point-wise a.e. in $\Om $ as $k \to +\infty$. Since $v_k\geq w_\la$ for each $k$,  $v_\la\geq w_\la$ a.e. in $\Om$.\\ {\bf Claim:}  $v_\la$ is a weak solution to \eqref{pp}.\\ For $\phi \in H_0^1(\Om)$ and $\e>0$, we set $$\psi_{k,\e} := (v_k+\e\phi-w_\la)^-,\;\; \psi_\e:= (v_\la+\e\phi -w_\la)^-.$$  Clearly, $\psi_{k,\e}\in H_0^1(\Om)$. This gives  that $(v_k +\e\phi+\psi_{k,\e}) \in T$. Taking $z = v_k +\e\phi+\psi_{k,\e}$ in \eqref{sec-sol-gm2}, we deduce
				\begin{equation}\label{sec-sol-gm3}
					\begin{split}
						-\frac{1}{k}\|(\e\phi+\psi_{k,\e})\|&\leq\int_\Om \nabla v_k \nabla (\e\phi+\psi_{k,\e}) -\la  \int_{\Om}\al(x)h(v_k)^{-q}h'(v_k)(\e\phi+\psi_{k,\e})~dx\\
						&\quad \quad- \la\int_\Om f(x,h(v_k))h'(v_k)(\e\phi+\psi_{k,\e})~dx.
					\end{split}
				\end{equation}
				Note that $|\psi_{k,\e}| \leq w_\la +\e|\phi|$. Hence, using  the Sobolev embedding and the Lebesgue dominated convergence theorem,  as $k \to +\infty$, $\psi_{k,\e} \to \psi_\e$  in  $L^p(\Om),  \;p \in [1,+\infty)$;  $  \psi_{k,\e} \rightharpoonup \psi_\e $  weakly in $H^1_0(\Om)$ and 
				$\psi_{k,\e}(x) \to \psi_\e(x)$	 point-wise a.e. in $\Om$.
				Now using Lemma \ref{L1}-$(h_{13})$, we obtain
				\begin{align*}
					\al(x)h(v_k)^{-q}h'(v_k)(\e\phi+\psi_{k,\e})\leq\|\al\|_\infty h(v_\la)^{-q}h'(v_\la)(v_\la+2\e|\phi|).
				\end{align*} Furthermore, using the fact that both $h$ and $ f(x,\cdot)$  are non decreasing functions, we get 
				\begin{align*}
					f(x,h( v_k))h'(v_k)(\e\phi+\psi_{k,\e})\leq f(x,h(w_\la+\e|\phi|))(w_\la+\e|\phi|).
				\end{align*}
				Hence, employing the Lebesgue dominated convergence theorem, as $k\to+\infty$, we infer
				\begin{align}\label{sg}
					\int_\Om \al(x) h(v_k)^{-q}h'(v_k)(\e\phi+\psi_{k,\e}) ~dx &\to \int_\Om \al(x) h(v_\la)^{-q}h'(v_\la)(\e\phi+\psi_{\e}) ~dx,\\
					\int_\Om f(x,h( v_k))h'(v_k)(\e\phi+\psi_{k,\e}) ~dx &\to \int_\Om f(x,h( v_\la))h'(v_\la)(\e\phi+\psi_{\e}) ~dx.\label{ct}
				\end{align}
				We define the sets $$\Om_{k,\e}:= \text{supp}\;\psi_{k,\e},\;\;\; \Om_\e:= \text{supp}\;\psi_\e \text\;\; \text{ and \;\;} \Om_0 :=\{x \in \Om:\; v_\la(x):=w_\la(x)\}.$$ Then,  \begin{align}\label{ms1} |\Om_\e\setminus \Om_0| & \to 0 \text{\;\;\;  as }\e \to 0;\\ |\Om_{k,\e}\setminus \Om_\e|+ |\Om_\e\setminus\Om_{k,\e}| & \to 0 \text{ \;\; as } k \to +\infty.\label{ms2}
				\end{align} Therefore, as $k\to+\infty$, 
				\begin{align}\label{gd}
					\int_\Om \nabla v_k\nabla \psi_{k,\e} ~dx &=\int_{\Om_{\e}} \nabla v_k\nabla \psi_{\e}~dx
					-\int_{\Om_{k,\e}} |\nabla( v_k-v_\la)|^2 ~dx+\int_{\Om_{k,\e}} \nabla v_\la \nabla( v_\la-v_k)~dx+o(1)\notag\\
					&\leq \int_{\Om} \nabla v_k\nabla \psi_{\e}~dx
					+\int_{\Om_{\e}} \nabla v_\la \nabla( v_\la-v_k)~dx +o(1)=\int_{\Om} \nabla v_k\nabla \psi_{\e}~dx+o(1).\end{align} 
				Combining  \eqref{sec-sol-gm3}, \eqref{sg}, \eqref{ct}, \eqref{ms1}, \eqref{ms2} and \eqref{gd}, letting $k\to +\infty$ and using Lemma \ref{L1}-$(h_{13})$, we obtain
				\begin{align}\label{ee1}
					&\int_\Om\nabla v_\la \nabla \phi ~dx -\la \int_\Om \al(x) h(v_\la)^{-q}h'(v_\la)\phi ~dx-\la\int_\Om f(x,h( v_\la))h'(v_\la)\phi ~dx\notag\\
					&\geq-\frac 1\e \left[\int_\Om\nabla v_\la \nabla \psi_\e ~dx - \la\int_\Om \al(x) h(v_\la)^{-q}h'(v_\la)\psi_{\e} ~dx-\la\int_\Om f(x,h( v_\la))h'(v_\la)\psi_{\e} ~dx\right]\notag\\
					&=\frac 1\e \bigg[\int_\Om -\nabla (v_\la-w_\la )\nabla \psi_\e ~dx +\la\int_\Om \al(x) \left[h(v_\la)^{-q}h'(v_\la)-h(w_\la)^{-q}h'(w_\la)\right]\psi_{\e} ~dx\notag\\ &\qquad\qquad\qquad+\la\int_\Om\left[f(x,h( v_\la))h'(v_\la)-f(x,h( w_\la))h'(w_\la)\right]\psi_{\e} ~dx\bigg]\notag\\
					&=\frac 1\e \bigg[\int_{\Om_\e }-\nabla (v_\la-w_\la )\nabla (w_\la-v_\la-\e\phi) ~dx \notag\\&\qquad\qquad\quad+\la\int_{\Om_\e }\al(x) \left[h(v_\la)^{-q}h'(v_\la)-h(w_\la)^{-q}h'(w_\la)\right](w_\la-v_\la-\e\phi) ~dx\notag\\ &\qquad\qquad\qquad\qquad+\la \int_{\Om_\e}\left[f(x,h( v_\la))h'(v_\la)-f(x,h( w_\la))h'(w_\la)\right](w_\la-v_\la-\e\phi) ~dx\bigg]\notag\\
					&\geq\int_{\Om_\e }\nabla (v_\la-w_\la )\nabla \phi ~dx -\la\int_{\Om_\e }\al(x) \left[h(v_\la)^{-q}h'(v_\la)-h(w_\la)^{-q}h'(w_\la)\right]\phi ~dx \notag\\ &\quad\;\;-\la\int_{\Om_\e}\left[f(x,h( v_\la))h'(v_\la)-f(x,h( w_\la))h'(w_\la)\right]\phi ~dx\notag\\&\qquad\qquad+\frac{\la}{\e}\int_{\Om_\e}\left[f(x,h( v_\la))h'(v_\la)-f(x,h( w_\la))h'(w_\la)\right](w_\la-v_\la) ~dx.
				\end{align}
				Since $v_\la\geq w_\la,$ using Lemma \ref{L1}-$(h_9)$ and the mean value theorem, if $x\in \Om_\e$,
				\begin{align}\label{fa}
					&\left[f(x,h( v_\la))h'(v_\la)-f(x,h( w_\la))h'(w_\la)\right](w_\la-v_\la) ~dx\notag\\ &\qquad\geq -(v_\la-w_\la)^2\left[f'(x,h(\xi_\la))(h'(\xi_\la))^2+f(x,h(\xi_\la))h''(\xi_\la)\right], \;\;\;\;\; \xi_\la\in (w_\la, v_\la)\notag\\
					&\qquad\geq -\e^2 f'(x,h(\xi_\la)) \phi^2.
				\end{align} Plugging \eqref{fa} into \eqref{ee1}, letting $\e\to 0^+$ and using \eqref{ms1}, we get
				\begin{align*}
					&\int_\Om\nabla v_\la \nabla \phi ~dx -\la \int_\Om \al(x) h(v_\la)^{-q}h'(v_\la)\phi ~dx-\la\int_\Om f(x,h( v_\la))h'(v_\la)\phi ~dx\\&\quad\quad\geq o(1)-\la \e\int_{\Om_\e}f'(x,h(\xi_\la)) \phi^2~dx=o(1).
				\end{align*} Considering $-\phi$ in place of $\phi$ and arguing similarly as above, we get the reverse inequality in the last relation. Therefore,
				\begin{align*}
					\int_{\Om} \nabla v_\la\nabla \phi ~dx-\la\int_\Om \al(x)h(v_\la)^{-q} h'(v_\la)\phi ~dx-\la\int_{\Om}f(x,h(v_\la))h'(v_\la)\phi ~dx=0
				\end{align*} for all $\phi\in H^1_0(\Om).$ So, $v_\la\in H^1_0(\Om)$ is a weak solution to \eqref{pp} and thus,  the claim is proved. Moreover, by Lemma \ref{lem6}, $v_\la\in C^+_{\varphi_{q}}(\Om)$. \\
				Now we show that $v_\la\not= w_\la.$ For that, it is enough to prove that 
				\begin{align}\label{cv1}
					v_k\to v_\la \text{ \;\; in \; } H^1_0(\Om) \text{\;\; as } k\to+\infty.
				\end{align}
				Applying the Br\'ezis-Lieb lemma, we get
				\begin{align*}
					\|v_k\|^2-\|v_k-v_\la\|^2 &=\|v_\la\|^2+o(1).
				\end{align*} Putting $z=v_\la$ in \eqref{sec-sol-gm2} and
				using the fact that $v_k\rightharpoonup v_\la$ in $T$ as $k\to+\infty$, we  obtain that
				\begin{equation}\label{sec-sol-gm4}
					\int_\Om|\nabla( v_k-v_\la)|^2 dx\leq	o(1)-\la\int_\Om\al(x) h(v_k)^{-q}h'(v_k)(v_\la-v_k)~dx -\la\int_\Om f(x,h(v_k))h'(v_k)(v_\la-v_k)~dx.
				\end{equation} Let us denote $u_k:=\frac{v_k-w_\la}{\|v_k-w_\la\|}$.
				Now  recalling $(f2)$ and \eqref{f} for any $\e>0$, Lemma \ref{L1}-$(h_3), (h_4), (h_5)$ and using  H\"older's inequality, for some large $N>>1$, we deduce
				\begin{align}\label{gm5}
					&\int_{\Om\cap\{x\,:\, v_k(x)>N\}} f(x,h(v_k))h'(v_k)v_k~dx\notag\\&\leq C \int_{\Om\cap\{x\,:\, v_k(x)>N\}} \exp \left((1+\e)h(v_k)^4\right) h'(v_k)v_k~dx\notag\\
					&\leq C \int_{\Om\cap \{x\,:\, v_k(x)>N\}} \exp \left(2(1+\e) v_k^2\right) v_k~dx\notag\\
					&\leq C \int_{\Om\cap \{x\,:\, v_k(x)>N\}} \exp \left(3(1+\e) v_k^2\right) ~dx\notag\\
					&= C \int_{\Om\cap \{x\,:\, v_k(x)>N\}} \exp \left(-(1+\e) v_k^2\right) \exp \left(4(1+\e) (v_k-w_\la+w_\la)^2\right) ~dx\notag\\
					&\leq C \exp \left(-(1+\e) N^2\right)\int_{\Om\cap\{x\,:\, v_k(x)>N\}} \exp \left(8(1+\e) [(v_k-w_\la)^2+w_\la^2]\right) ~dx\notag\\
					&\leq C \exp \left(-(1+\e) N^2\right)\int_{\Om\cap\{x\,:\, v_k(x)>N\}} \exp \left(8(1+\e) u_k^2\|v_k-w_\la\|^2\right)\exp\left(8(1+\e)w_\la^2\right) ~dx\notag\\
					&\leq C \exp \left(-(1+\e) N^2\right) \left(\int_{\Om\cap\{x\,:\, v_k(x)>N\}} \exp \left(8p(1+\e) u_k^2\|v_k-w_\la\|^2\right)~dx \right)^{1/p}\notag\\&\qquad\qquad\qquad\qquad\qquad\qquad \left(\int_\Om\exp\left(8p'(1+\e)w_\la^2\right) ~dx\right)^{1/p'}\notag\\
					&\leq C  \exp \left(-(1+\e) N^2\right) \bigg(\int_{\Om\cap\{x\,:\, v_k(x)>N\}} \exp \left(8p(1+\e) u_k^2(\sigma+r)^2\right)~dx \bigg)^{1/p}\notag\\&\qquad\qquad\qquad\qquad\qquad\qquad\left(\int_\Om\exp\left(8p'(1+\e)w_\la^2\right) ~dx\right)^{1/p'}.
				\end{align} Now choosing $p>1$  and $\sigma_0>0$ appropriately small so that $8(1+\e)p(\sigma+r)<8(1+\e)p \sigma_0<4\pi$ and then  applying Theorem \ref{TM-ineq}, from \eqref{gm5}, we derive \begin{align}\label{gd0}\int_{\Om\cap\{x\,:\, v_k(x)>N\}} f(x,h(v_k))h'(v_k)v_k~dx=O\left(\exp \left(-(1+\e) N^2\right)\right).
				\end{align} Since $f(x,h(v_k))h'(v_k)v_k\to f(x,h(v_\la))h'(v_\la)v_\la$ point-wise a.e. in $\Om$ as $k\to+\infty$, using \eqref{gd0} and applying the Lebesgue dominated convergence theorem, it follows that \begin{align}\label{gd1}\int_\Om f(x,h(v_k))h'(v_k)v_k ~dx&=\int_{\Om\cap\{x\,:\, v_k\leq N\}} f(x,h(v_k))h'(v_k)v_k ~dx+\int_{\Om\cap\{x\,:\, v_k>N\}} f(x,h(v_k))h'(v_k)v_k ~dx\notag\\
					&= \int_{\Om\cap\{x\,:\, v_k\leq N\}} f(x,h(v_k))h'(v_k)v_k ~dx+O\left(\exp \left(-(1+\e) N^2\right)\right)\notag\\
					&\to\int_\Om f(x,h(v_\la))h'(v_\la)v_\la ~dx \text {\;\;\;\; as }  k\to +\infty \text{\; and } N\to+\infty.
				\end{align} In a similar way, we also have  $\ds\int_\Om f(x,h(v_k))h'(v_k)v_\la~dx\to\int_\Om f(x,h(v_\la))h'(v_\la)v_\la ~dx$ as $k\to +\infty.$ This, together with \eqref{gd1}, yields that
				\begin{align}\label{fcv}
					\int_\Om f(x,h(v_k))h'(v_k)(v_k-v_\la) ~dx\to 0 \text{  \;\;\; as } k\to+\infty.
				\end{align} Next, we show that \begin{align}\label{r2}
					\int_\Om \al(x)  h(v_k)^{-q}h'(v_k)(v_k-v_\la) ~dx \to 0 \text{\;\;\; as } k\to+\infty.
				\end{align}   
				By the construction, we have $v_k, v_\la\geq w_\la\geq \uline {w_\la}$. Next, we consider the three cases separately as following:\\
				{\bf Case I: $0<q<1$.} Using  Lemma \ref{L1}-$(h_3), (h_8)$, \eqref{inqq1}  with $c_1(q,\la)>0$ small enough such that $c_1(q,\la)\de<1$ and \eqref{Inqq1} for $v_\la\in$, we get \begin{align}\label{r1} \al(x)h(v_k)^{-q}h'(v_k)v_\la\leq \al(x)h(\uline{w_\la})^{-q} v_\la\leq \al(x)h(c_1(q,\la)\de)^{-q}C_2(q,\la)\de\leq \|\al\|_\infty C(q,\la) h(1)^{-q}  \de^{1-q} \in L^1(\Om).
				\end{align} 
				Moreover, by Lemma \ref{L1}-$ (h_4)$, $(h_5)$, it follows that \begin{align}\label{r3}
					\al(x)	h(v_k)^{-q}h'(v_k)v_k \leq \al(x) h(v_k)^{1-q}\leq \|\al\|_\infty v_k^{1-q}\in L^1(\Om)
				\end{align} due to the Sobolev inequality and the Hölder inequality. Now using  \eqref{hs},  $\int_\Om v_k^{1-q} dx\to   \int_\Om v_\la^{1-q} dx$ as $k\to+\infty.$
				So, by the Lebesgue dominated convergence theorem, as $k\to+\infty$, combining   \eqref{r1} and \eqref{r3}, we get  \eqref{r2}.\\
				{{{\bf Case II: $q=1$.} Again using Lemma \ref{L1}-$(h_{13}), (h_3), (h_8)$,   \eqref{inqq3} with $c(\la)>0$ small enough such that $c(\la)\de<1$ and \eqref{Inqq3} for $v_\la$ with any small $0<\e<1$,    
						\begin{align}\label{r04}
							h(v_k)^{-1}h'(v_k)v_\la\leq h(\uline{w_\la})^{-1} h'(\uline{w_\la}) v_\la \leq  h(c(\la)\de) C_\e(\la) \de^{1-\e}\leq c(\la) C_\e(\la) h(1)^{-1} \de^{-1} \de^{1-\e}  =C(\la,\e) \de ^{-\e} \in L^1(\Om).
						\end{align}From this  and the hypothesis $(\al1)$, \eqref{r2} follows, thanks to the  Lebesgue dominated convergence theorem.}}\\
				{\bf Case III: $1<q<3$.}   In view of \eqref{inqq2}   with $c_1(q,\la)>0$ small enough such that $c_1(q,\la){\de}^{\frac{2}{1+q}}<1$, recalling \eqref{Inqq2} for $v_\la$ and using Lemma \ref{L1}-$(h_{13}), (h_3), (h_8)$, it yields that
				\begin{align}\label{r4}
					h(v_k)^{-q}h'(v_k)v_\la\leq h(\uline{w_\la})^{-q} h'(\uline{w_\la})v_\la\leq h\left(c_1{\de}^{\frac{2}{1+q}}\right)^{-q} v_\la \leq (c_1h(1))^{-q} {\de}^{\frac{-2q}{1+q}} C_2(q,\la) {\de}^{\frac{2}{1+q}}\leq C(q,\la) \de ^{\frac{2(1-q)}{1+q}} \in L^1(\Om)
				\end{align} since $\frac{2(1-q)}{1+q}>-1.$
				Again, repeating a similar argument as in \eqref{mon1}, it follows that \begin{align}\label{r5}
					h(v_k)^{-q}h'(v_k)v_k \leq h (\uline{w_\la})^{1-q}\in L^1(\Om).
				\end{align} Thus, \eqref{r4} and \eqref{r5}, combining with the hypothesis $(\al1)$ and the Lebesgue dominated convergence theorem, yield \eqref{r2}.\\
				Therefore, taking into account \eqref{fcv} and \eqref{r2}, from \eqref{sec-sol-gm4}, we finally get \eqref{cv1}. This completes the proof of the proposition.
			\end{proof}
			
			\noi Now we prove the existence of second solution in the case where $(MP)$ occurs.\\
			
			Now for $k\in \mb N$, we  define the  Moser function $\mc M_k:\ \Om\to\mb R$ as
			\begin{equation*}
				\mc{ {M}}_{k}(x)=\frac{1}{\sqrt{2\pi}}\left\{
				\begin{array}{ll}
					& \ds(\log k)^{\frac{1}{2}} \;\;\;\;\; \text{\;\; if \;} 0\leq |x|\leq \frac{1}{k},\\
					& \ds\frac{\log \left(\frac{1}{|x|}\right)}{(\log k)^{\frac{1}{2}}} \quad\text{\;\; if \;} \; \frac{1}{k}\leq |x|\leq 1,\\
					&\quad 0 \;\;\;\qquad\;\; \text{\;\; if \;} |x|\geq 1.
				\end{array}
				\right.
			\end{equation*}
			Then, $	\mc{{M}}_{k}\in H^{1}_0(\Om)$ and supp $ \mc{{M}}_{k} \subseteq {B_1}$. Moreover, we define the function $$\mc M_k^{\ell}(x):=\mc M_k\left(\frac{x-x_0}{\ell}\right).$$ We choose $x_0$ and $\ell$ is such a way so that $\text{supp}\, \mc M_k^{\ell}\subset \Om$. Note that $\|\mc M_k^{\ell}\|=1.$
			
			\begin{lemma}\label{minimizer-gm}Let the conditions in Theorem \ref{thm2} and $(MP)$ hold. Then
				\begin{enumerate}
					\item[$(i)$] $J_\la(w_\la+ t \mc M_k^\ell) \to-\infty$ as $t\to+\infty$ uniformly for $k$ large.
					\item[$(ii)$]  $\ds\sup_{t\geq 0}J_\la(w_\la+ t \mc M_k^\ell)<J_\la(w_\la)+\pi$ for large $k$.
				\end{enumerate}
			\end{lemma}
			\begin{proof} We prove this lemma for the case $q\not=1$. When $q=1$, the proof follows similarly.\\ 
				\noi {  $(i).$} By $(f2),$ there exist two positive constants $C_1,C_2$ such that $$F(x,s)\geq C_1 \exp (h(s)^4) -C_2 $$ for all $x\in \Om,\; s\geq0.$ Using this combining with Lemma \ref{L1}-$(h_8)$ for large $t>>1$ and the H\"older's inequality, we get
				\begin{align}
					&J_\la (	w_\la+ t \mc M_k^\ell)\notag\\&=\frac 12\int_\Om |\nabla (w_\la+ t \mc M_k^\ell)|^2 ~dx-\frac{\la}{1-q}\int_\Om \al(x) h(	w_\la+ t \mc M_k^\ell)^{1-q} ~dx- \la \int_\Om F(x, h(	w_\la+ t \mc M_k^\ell)) ~dx\notag\\
					&\leq \frac 12 \|w_\la\|^2+t\|w_\la\|\| \mc M_k^\ell\|+\frac 12 t^2 \|\mc M_k^\ell\|^2+C_2 \la |\Om|-C_1 \la\int_\Om \exp \left(  h(	w_\la+ t \mc M_k^\ell)^4\right)\notag\\
					&\leq  t^2 +C_3(\la) t-\la C_1\int_{B_{\frac \ell k}(x_0)}  \exp \left(  h(1)^4	(w_\la+ t \mc M_k^\ell)^2\right) ~dx\notag\\&\leq  t^2 +C_3(\la) t- C_4(\la)\int_{B_{\frac \ell k}(x_0)}  \exp \left(t^2\frac { h(1)^4}{2\pi} \log k\right) ~dx\notag\\
					&= t^2 +C_3(\la) t- C_4(\la) ~  k^{ \frac{h(1)^4}{2\pi} t^2-2}\to-\infty \;\;\;\text{uniformly in $k$ as $t\to+\infty$}.\notag
				\end{align}
				\noi { $(ii).$} Suppose $(ii)$ does not hold. So, there exists a sub-sequence of $\{\mc M_k^\ell\}$, still denoted by $\{\mc M_k^\ell\}$, such that 
				\begin{align}\label{l1}
					\max_{t\geq 0} J_\la (w_\la+ t \mc M_k^\ell)\geq J_\la(w_\la)+\pi.
				\end{align}
				Then,
				\begin{align}\label{l2}
					J_\la(w_\la+ t \mc M_k^\ell)&=\frac 12\|w_\la\|^2+t
					\int_\Om \nabla w_\la \nabla \mc M_k^\ell ~dx+\frac{t^2}{2}-\frac{\la}{1-q}\int_\Om \al(x) h(	w_\la+ t \mc M_k^\ell)^{1-q} ~dx\notag\\&\qquad\qquad- \la \int_\Om F(x, h(	w_\la+ t \mc M_k^\ell)) ~dx\notag\\
					&= J_\la (w_\la)+\frac{\la}{1-q}\int_\Om \al(x) h(	w_\la)^{1-q} ~dx+ \la \int_\Om F(x, h(	w_\la) )~dx\notag\\&\qquad+t{\la}\int_\Om \al(x) h(	w_\la)^{-q} h'(w_\la)  \mc M_k^\ell ~dx+ t \la \int_\Om f(x, h(	w_\la)) h'(w_\la)\mc M_k^\ell ~dx\notag\\& \qquad\quad+\frac{t^2}{2}-\frac{\la}{1-q}\int_\Om \al(x) h(	w_\la+ t \mc M_k^\ell)^{1-q} ~dx- \la \int_\Om F(x, h(	w_\la+ t \mc M_k^\ell)) ~dx.
				\end{align} From $(i)$ and \eqref{l1}, it follows that there exists $t_k\in (0,+\infty)$ and $L_0>0$ satisfying $t_k\leq L_0$,  such that
				\begin{align}\label{l3}
					\max_{t\geq 0} J_\la (w_\la+ t \mc M_k^\ell)= J_\la (w_\la+ t_k \mc M_k^\ell)\geq J_\la(w_\la)+\pi.
				\end{align} 
				Now \eqref{l2} and \eqref{l3} imply that
				\begin{align}
					&\frac{\la}{1-q}\int_\Om \al(x) h(	w_\la)^{1-q} ~dx + \la\int_\Om F(x, h(	w_\la)) ~dx\notag\\&\qquad+t_k{\la}\int_\Om \al(x) h(	w_\la)^{-q} h'(w_\la)  \mc M_k^\ell ~dx+ t_k \la \int_\Om f(x, h(	w_\la)) h'(w_\la)\mc M_k^\ell ~dx\notag\\& \qquad\quad+\frac{t_k^2}{2}-\frac{\la}{1-q}\int_\Om \al(x) h(	w_\la+ t_k \mc M_k^\ell)^{1-q} ~dx- \la \int_\Om F(x, h(	w_\la+ t_k \mc M_k^\ell)) ~dx\geq \pi.\notag
				\end{align} Using mean value theorem and Lemma \ref{L1}-$(h_9)$, we deduce
				\begin{align}\label{l5}
					&\frac{\la}{1-q}\left[\int_\Om \al(x) h(	w_\la)^{1-q} ~dx-\int_\Om \al(x) h(	w_\la+ t_k \mc M_k^\ell)^{1-q} ~dx\right] +t_k{\la}\int_\Om \al(x) h(	w_\la)^{-q} h'(w_\la)  \mc M_k^\ell ~dx\notag\\
					&=-\la t_k\left[\int_\Om \al(x) h(	w_\la+ t_k\theta_k \mc M_k^\ell)^{-q} h'(w_\la+ t_k\theta_k \mc M_k^\ell)  \mc M_k^\ell ~dx-\int_\Om \al(x) h(	w_\la)^{-q} h'(w_\la)  \mc M_k^\ell ~dx\right]\notag\\
					&=-\la t_k^2\theta_k\int_\Om \al(x) {|\mc M_k^\ell|^2}\Big[-qh(w_\la+ t_k\xi_k \mc M_k^\ell)^{-q-1}h'(w_\la+ t_k\xi_k \mc M_k^\ell)\notag\\&\qquad\qquad\qquad\qquad\qquad\qquad\qquad+h(w_\la+ t_k\xi_k \mc M_k^\ell)^{-q}h''(w_\la+ t_k\xi_k \mc M_k^\ell)\Big]\notag\\
					&=\la t_k^2\theta_k\int_\Om \al(x) {|\mc M_k^\ell|^2}\Big[qh(w_\la+ t_k\xi_k \mc M_k^\ell)^{-q-1}h'(w_\la+ t_k\xi_k \mc M_k^\ell)\notag\\&\qquad\qquad\qquad\qquad\qquad+2 h(w_\la+ t_k\xi_k \mc M_k^\ell)^{1-q}(h'(w_\la+ t_k\xi_k \mc M_k^\ell))^4\Big],
				\end{align} where $\theta_k,\xi_k\in(0,1)$. Now by Lemma \ref{L1}-$(h_3),(h_{12})$,
				\[2h(s)^{1-q}(h'(s))^4=2h(s)^{-1-q}(h(s)h'(s))^2(h'(s))^2\leq h(s)^{-1-q}h'(s).\] Plugging the last relation in \eqref{l5} and  using Lemma \ref{L1}-$(h_{13}), (h_3), (h_8)$ together with the fact that \\$w_\la+ t_k\xi_k \mc M_k^\ell\geq \uline{ w_\la}>0 $ in $B_{\frac{\ell}{k}}(x_0)$, we obtain
				\begin{align}\label{l6}
					&\frac{\la}{1-q}\left[\int_\Om \al(x) h(	w_\la)^{1-q} ~dx-\int_\Om \al(x) h(	w_\la+ t_k \mc M_k^\ell)^{1-q} ~dx\right] +t_k{\la}\int_\Om \al(x) h(	w_\la)^{-q} h'(w_\la)  \mc M_k^\ell ~dx\notag\\
					&\leq \la t_k^2\theta_k (q+1)\int_\Om \al(x) {|\mc M_k^\ell|^2} h(w_\la+ t_k\xi_k \mc M_k^\ell)^{-q-1}h'(w_\la+ t_k\xi_k \mc M_k^\ell) ~dx\notag\\
					&\leq \la t_k^2\theta_k (q+1)\|\al\|_\infty\int_\Om {|\mc M_k^\ell|^2} h(\uline{ w_\la})^{-q-1}h'(\uline{ w_\la}) ~dx\notag\\
					&\leq \la t_k^2\theta_k (q+1)\|\al\|_\infty(h(1))^{-1-q}\int_\Om{|\mc M_k^\ell|^2} \uline {w_\la}^{-q-1}~dx
					=t_k^2 O\left(\frac {1} {\log k}\right).
				\end{align} On the other hand, again using the mean value theorem and Lemma \ref{L1}-$(h_9), (h_{12}),(h_3)$, we infer 
				\begin{align}\label{l7}
					&t_k\int_\Om f(x, h(	w_\la))h'(w_\la)\mc M_k^\ell ~dx+\int_\Om F(x,h(w_\la))~dx-\int_\Om F(x, h(	w_\la+ t_k \mc M_k^\ell)) ~dx\notag\\
					&=t_k\left[\int_\Om f(x, h(	w_\la))h'(w_\la)\mc M_k^\ell ~dx-\int_\Om f(x, h(	w_\la+t_k \tilde\theta_k\mc M_k^\ell))h'(w_\la+t_k \tilde\theta_k\mc M_k^\ell)\mc M_k^\ell ~dx\right] \notag\\
					&=t_k^2\int_\Om -\tilde\theta_k |\mc M_k^\ell|^2 \Big[f'(x, h(	w_\la+t_k \tilde\xi_k\mc M_k^\ell))(h'(w_\la+t_k \tilde\xi_k\mc M_k^\ell))^2\notag\\&\qquad\qquad+f(x, h(	w_\la+t_k \tilde\xi_k\mc M_k^\ell))h''(w_\la+t_k \tilde\xi_k\mc M_k^\ell)\Big] ~dx\notag\\
					&=t_k^2\tilde\theta_k\int_\Om  |\mc M_k^\ell|^2 \bigg[-f'(x, h(	w_\la+t_k \tilde\xi_k\mc M_k^\ell))(h'(w_\la+t_k \tilde\xi_k\mc M_k^\ell))^2\notag\\&\qquad\qquad\qquad\qquad\qquad+2f(x, h(	w_\la+2t_k \tilde\xi_k\mc M_k^\ell))h(w_\la+t_k \tilde\xi_k\mc M_k^\ell)(h'(w_\la+t_k \tilde\xi_k\mc M_k^\ell))^4\bigg] ~dx\notag\\
					&\leq t_k^2\tilde\theta_k\int_\Om  |\mc M_k^\ell|^2 \bigg[-f'(x, h(	w_\la+t_k \tilde\xi_k\mc M_k^\ell))(h'(w_\la+t_k \tilde\xi_k\mc M_k^\ell))^2\notag\\&\qquad\qquad\qquad\qquad\qquad+\sqrt{2}f(x, h(	w_\la+2t_k \tilde\xi_k\mc M_k^\ell))h(w_\la+t_k \tilde\xi_k\mc M_k^\ell)h'(w_\la+t_k \tilde\xi_k\mc M_k^\ell)\bigg] ~dx\notag\\
					&\leq  t_k^2\tilde\theta_k\int_\Om  |\mc M_k^\ell|^2 \bigg[(\sqrt{2}-M_0)f'(x, h(	w_\la+t_k \tilde\xi_k\mc M_k^\ell))+L\bigg] h'(w_\la+t_k \tilde\xi_k\mc M_k^\ell) ~dx \;\;\;\;\qquad\qquad[\text{by \;} \eqref{rmk2}]\notag\\
					&\leq L t_k^2\tilde\theta_k\int_\Om  |\mc M_k^\ell|^2 ~dx=t_k^2 O\left(\frac{1}{\log k}\right),
				\end{align} where $\tilde\theta_k$, $\tilde\xi_k\in(0,1).$
				Now gathering \eqref{l1}, \eqref{l6} and \eqref{l7}, it follows that
				$
				\frac{t_k^2}{2} + O\left(\frac{1}{\log k}\right) t_k^2\geq\pi.
				$ That is, 
				\begin{align}\label{l8}
					{t_k^2} \geq 2\pi- O\left(\frac{1}{\log k}\right).
				\end{align} In virtue of  the relation $\frac{d}{dt}J_\la(w_\la+t\mc M_k^\ell)|_{t=t_k}=0,$ we have
				\begin{align}\label{l10}
					t_k^2+t_k\int_\Om \nabla w_\la \nabla \mc M_k^\ell ~dx&-{\la}\int_\Om \al(x) h(	w_\la+ t_k \mc M_k^\ell)^{-q} h'(	w_\la+ t_k \mc M_k^\ell) t_k \mc M_k^\ell ~dx\notag\\&\qquad\qquad\qquad= \la\int_\Om f(x, h(	w_\la+ t_k \mc M_k^\ell)) h'(	w_\la+ t_k \mc M_k^\ell) t_k \mc M_k^\ell ~dx.
				\end{align} Now we estimate the  right hand side in \eqref{l10}.\\
				Using the fact that $w_\la\geq C>0$ on $B_{\frac{\ell}{k}}(x_0)$ combining with $(f1)$ and Lemma \ref{L1}-$(h_4)$, we obtain
				\begin{align}\label{l11}
					&\int_\Om f(x, h(	w_\la+ t_k \mc M_k^\ell)) h'(	w_\la+ t_k \mc M_k^\ell) t_k \mc M_k^\ell ~dx\notag\\
					&\geq \int_{B_{\frac{\ell}{k}}(x_0)} f(x, h(	C+ t_k \mc M_k^\ell)) h'(	C+ t_k \mc M_k^\ell) t_k \mc M_k^\ell ~dx\notag\\
					&\geq \frac 12\int_{B_{\frac{\ell}{k}}(x_0)} f(x, h(	C+ t_k \mc M_k^\ell)) h(	C+ t_k \mc M_k^\ell) \frac{ t_k \mc M_k^\ell}{C+ t_k \mc M_k^\ell} ~dx.
				\end{align}
				By $(f1)$, since $g(x,s)$ is non decreasing in $s$, we get
				\begin{align}\label{f7}
					\lim_{s\to +\infty}\frac{sf(x,s)}{\exp(s^4)}=	\lim_{s\to +\infty}{sg(x,s)}=+\infty.
				\end{align} Therefore, for any $b>0$ there exists some constant $M_b>0$ such that
				\begin{align}\label{l12}
					f(x, h(	C+ t_k \mc M_k^\ell)) h(	C+ t_k \mc M_k^\ell)> b\exp\left(h(	C+ t_k \mc M_k^\ell)^4\right) \text{\;\;\;\;\;\;\;\; for all } t_k>M_b.
				\end{align}
				On the other hand, by Lemma \ref{L1}-$(h_7)$, for any $\e>0$ there exists $m_\e>0$ such that 
				\begin{align}\label{l13}
					h(C+ t_k \mc M_k^\ell)^4>(C+ t_k \mc M_k^\ell)^2(2-\e) \;\;\;\;\;\;\text{for all } t_k\geq m_\e.
				\end{align} Plugging \eqref{l12} and \eqref{l13} in \eqref{l11} and using  \eqref{l8}, for $t_k>\max\{M_b,m_\e\}$, we get
				\begin{align}\label{l14}
					&\int_{B_{\frac{\ell}{k}}(x_0)} f(x, h(	w_\la+ t_k \mc M_k^\ell)) h'(	w_\la+ t_k \mc M_k^\ell) t_k \mc M_k^\ell ~dx\notag\\
					&\geq \frac b2\int_{B_{\frac{\ell}{k}}(x_0)} \exp(h(	C+ t_k \mc M_k^\ell)^4) \frac{ t_k \mc M_k^\ell}{C+ t_k \mc M_k^\ell} ~dx\notag\\
					&\geq \frac b2\int_{B_{\frac{\ell}{k}}(x_0)} \exp((	C+ t_k \mc M_k^\ell)^2(2-\e)) \frac{ t_k \mc M_k^\ell}{C+ t_k \mc M_k^\ell} ~dx\notag\\
					&\geq \frac b2 \frac{ t_k \mc M_k^\ell(x_0)}{C+ t_k \mc M_k^\ell(x_0)}\int_{B_{\frac{\ell}{k}}(x_0)} \exp((	C+ t_k \mc M_k^\ell)^2(2-\e)) ~dx\notag\\
					&= \frac b2 \frac{ t_k \mc M_k^\ell(x_0)}{C+ t_k \mc M_k^\ell(x_0)} \exp\left(\left[C^2+ 2Ct_k\mc M_k^\ell(x_0)+t_k^2 \mc |M_k^\ell(x_0)|^2\right](2-\e)\right) | B_{\frac{\ell}{k}}(x_0)|\notag\\
					&= \frac b2 \frac{ t_k \mc M_k^\ell(x_0)}{C+ t_k \mc M_k^\ell(x_0)} C_0 \exp\left(2Ct_k\mc M_k^\ell(x_0)+t_k^2 \mc |M_k^\ell(x_0)|^2\right) | B_{\frac{\ell}{k}}(x_0)|\;\;\; \text{\; as }\e\to 0^+\notag\\
					&=\frac b2 \frac{1}{1+\frac {C} {t_k \mc M_k^\ell(x_0)}}C_0 \pi \ell^2 \exp\left(2Ct_k\mc M_k^\ell(x_0)+{2\log k}\left(\frac {t_k^2}{2\pi} -1\right)\right) \notag\\
					&\geq \frac b 2 \frac{1}{1+\frac {C\sqrt{2\pi}} {t_k \sqrt{\log k}}}C_1 \pi \ell^2 \exp\left(2Ct_k\mc M_k^\ell(x_0)\right).
				\end{align} Now incorporating \eqref{l14} and \eqref{l11} in \eqref{l10} and using H\"older's inequality, we get 
				\[+\infty>C_2:=L_0^2+\|w_\la\|L_0\geq t_k^2+\|w_\la\| t_k\geq \frac b 2 \frac{\la}{1+\frac {C\sqrt{2\pi}} {t_k \sqrt{\log k}}}C_1 \pi \ell^2 \exp\left(2Ct_k\mc M_k^\ell(x_0)\right).\]  Since $\{t_k\}$ is bounded away from $0$, for some $C_3 > 0$, we have $t_k\mc M_k^\ell(x_0)>C_3(\log k)^{1/2}\to+\infty$ as $k\to+\infty$. Hence, letting $k\to+\infty$ in the last relation, it yields that
				\[C_2\geq \frac{b}{2} \la C_1\pi \ell^2 \exp (2CC_3
				\sqrt {\log k})\to+\infty.\] This is absurd since, $C_2<+\infty$. Hence, $(ii)$ follows.
				This completes the proof of the lemma.
			\end{proof}
			\noi Next, we recall the following result due to P. L. Lions   (\cite{Lions}).
			\begin{theorem}\label{hi}
				Let $\{u_k\,:\, \|u_k\| = 1\} $ be a sequence in $H^1_0(\Om)$ converging weakly in $H^1_0(\Om)$ to a non zero function 
				$u$. Then, for every $0<p <(1-\|u\|)^{-1}$
				\[\sup_k\int_\Om \exp(4\pi pu_k^2) ~dx<+\infty.\]
			\end{theorem}
			\begin{proposition}\label{prp2}
				Let the conditions in Theorem \ref{thm2} hold and let $\la\in(0,\La^*)$. Suppose that $(MP)$ holds. Then \eqref{pp} admits a  second solution $v_\la \in H^1_0(\Om) \cap C_{\varphi_q}^+(\Om)$ satisfying $v_\la\geq w_\la$ and $\|v_\la-w_\la\|=\sigma_1$.
			\end{proposition} 
			\begin{proof} 
				First,	we define the complete metric space
				\[\Gamma := \{\eta \in C([0,1],T):\; \eta(0)=w_\la,\; \|\eta(1)-w_\la\|>\sigma_1,\; J_\la(\eta(1))< J_\la(w_\la) \}\]
				with the metric  $$d(\eta,\tilde\eta)=\max\limits_{t\in[0,1]}\{\|\eta(t)-\tilde\eta(t)\|\} \text {\;\;\; for all }\eta,\tilde\eta \in \Gamma.$$ 
				Using Lemma \ref{minimizer-gm}-$(i)$, we have $\Gamma \neq \emptyset$. Let us set \begin{align*}\gamma_0 = \inf\limits_{\eta\in \Gamma}\max\limits_{t \in [0,1]}J_\la(\eta(t)).
				\end{align*} Then, in light of Lemma \ref{minimizer-gm}-$(ii)$  and the condition $(MP)$, we obtain
				\[J_\la(w_\la)< \gamma_0 \leq J_\la(w_\la)+\pi.\]
				Set $\Phi(\eta):= \max\limits_{t \in [0,1]}J_\la(\eta(t))$ for $\eta \in \Gamma$. Then employing Ekeland's variational principle to the functional $\Phi$ on $\Ga$, we get a sequence $\{\eta_k\}\subset \Gamma$ such that
				\begin{equation*}
					\Phi(\eta_k)<\gamma_0+ \frac{1}{k} \;\;\text{and}\;\;\;  \Phi(\eta_k){<} \Phi(\eta)+ \frac{1}{k}\|\Phi(\eta)-\Phi(\eta_k)\|_\Gamma\;\;\;\;\text {for all } \eta\in \Gamma.
				\end{equation*}
				Again, applying Ekeland's variational principle to $J_\la$ on $\Ga$  and arguing exactly similar to \cite[Lemma $3.5$]{badi}, one can show that there exists a sequence $\{v_k\}\subset T$ such that
				\begin{equation}\label{evp0}
					\left\{
					\begin{array}{ll}
						& J_\la(v_k) \to \ga_0\;\;\;\; \text{as \;}  k\to+\infty\\
						& \ds \int_{\Om} \nabla v_k\nabla(w-v_k)  ~dx-\la\int_\Om \al(x)h(v_k)^{-q} h'(v_k)(w-v_k) ~dx-\la\int_{\Om}f(x,h(v_k))h'(v_k)(w-v_k)~dx\\
						&\qquad\geq -\ds\frac C k (1+\|w\|)\;\;\;\text{for all}\; w \in T.
					\end{array}\right.
				\end{equation} Now choosing $w=2v_k$ in \eqref{evp0}, we get
				\begin{align}\label{j1}
					\int_{\Om}|\nabla v_k|^2 ~dx-\la \int_{\Om}\al(x)h(v_k)^{-q}h'(v_k) v_k ~dx-\la\int_\Om f(x,h(v_k))h'(v_k)v_k~dx\geq -\frac C k (1+\|2v_k\|).
				\end{align} In virtue of Lemma \ref{L1}-$(h_4)$, from \eqref{j1}, we obtain
				\begin{align}\label{j2}
					\int_{\Om}|\nabla v_k|^2 ~dx-\frac \la 2 \int_{\Om}\al(x)h(v_k)^{1-q}~dx-\frac \la2\int_\Om f(x,h(v_k))h(v_k)~dx\geq -\frac C k (1+\|2v_k\|).
				\end{align}
				We claim that  \begin{align}\label{bd}\ds\limsup_{k\to+\infty}\|v_k\|<+\infty.
				\end{align}
				For $0<q<1$,
				using \eqref{evp0}, \eqref{j2}, $(f3)$, Lemma \ref{L1}-$(h_5)$ and the  Sobolev embedding, as $k\to+\infty,$ it follows that
				\begin{align}\label{j3}
					\ga_0+\frac{4C}{\tau }\|v_k\|+o(1)&\geq \left(\frac 12-\frac 2\tau\right)\|v_k\|^2-\la\left(\frac{1}{1-q}-\frac 1\tau\right)\int_\Om\al(x)h(v_k)^{1-q} ~dx\notag\\&\qquad\qquad\qquad-\frac \la\tau \int_\Om [\tau F(x,h(v_k))-f(x,h(v_k))h(v_k)]~dx\\
					& \geq \left(\frac 12-\frac 2\tau\right)\|v_k\|^2-\la C(q,\Om)\left(\frac{1}{1-q}-\frac{1} {\tau}\right)\|\al\|_\infty \|v_k\|^{1-q} ~dx.\notag
				\end{align} Since $\tau>4$ and $2>1-q>0,$ the last relation yields that $\{v_k\}$ is bounded in $H^1_0(\Om)$ and hence,  \eqref{bd} holds.\\
				{Next, for  $q=1$,   using \eqref{j3} and  following the arguments in  similar way as in  \eqref{logg} and \eqref{loggg}, we obtain \eqref{bd}.}\\
				Lastly, for the case $1<q<3$, again using \eqref{j3} and  the exact arguments  in \eqref{d8} used for estimating the singular term, as $k\to+\infty$, we get \[\left(\frac 12-\frac 2\tau\right)\|v_k\|^2\leq (1+C)\|v_k\|+ \ga_0+o(1),\] which gives \eqref{bd}.  Therefore, there exists a $v_\la\in H^1_0(\Om) $ such that $v_k\rightharpoonup v_\la$ weakly in $H^1_0(\Om)$ and point-wise a.e. in $\Om$ as $k\to+\infty$. Arguing similarly  as in the proof of Proposition \ref{prp1}, we infer that $v_\la$ is a weak solution to \eqref{pp} and $v_\la\in H^1_0(\Om)\cap C_{\varphi_q}^+(\Om).$\\
				{\bf Claim: $v_\la\not=w_\la$.}\\
				In order to establish this  claim, it is sufficient to prove that $v_k\to v_\la$ in $H^1_0(\Om)$ as $k\to+\infty.$ For that, we need to establish 
				\begin{align}\label{compp}
					\int_\Om f(x,h(v_k)) h'(v_k) v_k ~dx\to \int_\Om f(x,h(v_\la)) h'(v_\la) v_\la ~dx \;\;\;\; \text{as } k\to+\infty.
				\end{align}Now borrowing the similar arguments as in \eqref{r2}, we have 
				\begin{align}\label{mpsi}
					\int_\Om \al(x)  h(v_k)^{1-q} dx \to \int_\Om \al(x)  h(v_\la)^{1-q} ~dx\text{\;\;\; as } k\to+\infty.
				\end{align}
				Next, we show that \begin{align}\label{mpF}
					\int_\Om F(x,h(v_k)) ~dx\to \int_\Om F(x,h(v_\la)) ~dx \;\;\;\; \text{as } k\to+\infty.
				\end{align} Since $\{v_k\}$ is bounded in $H^1_0(\Om)$, from \eqref{j2}, it follows that\begin{align}\label{o0}
					\limsup_{k\to+\infty}	\int_\Om f(x,h(v_k)) h'(v_k) v_k ~dx<+\infty.
				\end{align} Now using \eqref{o0} and Lemma \ref{L1}-$(h_8)$, for some large   $N>>1$, we  get 
				\begin{align*}\int_{\Om\cap\{x\,:\,h(v_k)(x)>N\}} f(x,h(v_k)) ~dx&\leq \frac 1N \int_{\Om\cap\{x\,:\,h(v_k)(x)>N\}} f(x,h(v_k)) h(v_k) ~dx\\&\leq \frac 2N \int_{\Om\cap\{x\,:\,h(v_k)(x)>N\}} 2 f(x,h(v_k)) h'(v_k) v_k\\&=O\left(\frac 1N\right).
				\end{align*} The last relation, combining with the Lebesgue dominated convergence theorem, yields
				\begin{align}\label{fof}\int_{\Om} f(x,h(v_k)) ~dx&=\int_{\Om\cap\{x\,:\,h(v_k)(x)\leq N\}} f(x,h(v_k)) ~dx+\int_{\Om\cap\{x\,:\,h(v_k)(x)>N\}} f(x,h(v_k)) ~dx\notag\\
					&=\int_{\Om\cap\{x\,:\,h(v_k)(x)\leq N\}} f(x,h(v_k)) ~dx+O\left(\frac 1N\right)\notag\\
					& \to\int_{\Om} f(x,h(v_\la)) ~dx,
				\end{align} as $k\to+\infty$ and $N\to+\infty$.
				Since by $(f4)$, 
				$F(x,h(v_k))\leq M_1(1+f(x,h(v_k)))$ for all $ k\in\mb N$,  using \eqref{fof} and the Lebesgue dominated convergence theorem, 
				\eqref{mpF} follows.
				Therefore, using \eqref{mpsi} and  \eqref{mpF} with the weak lower semicontinuity property of the norm, we derive
				\begin{align}\label{inl}
					J_\la(v_\la)\leq\liminf_{k\to+\infty} J_\la (v_k).
				\end{align} Supposing the contrary, let us assume that $\{v_k\}$ does not converge in $H^1_0(\Om)$ strongly. Then \eqref{mpsi} , \eqref{mpF} and \eqref{inl} imply that $ J_\la(v_\la)<\ga_0.$ By the hypothesis, $J_\la(w_\la)\leq J_\la(v_\la).$ So, for sufficiently small $\e>0$, by Lemma \ref{minimizer-gm}, we have
				\begin{align}\label{kk1}
					(\ga_0-J_\la(v_\la))(1+\e)\leq \left(\max_{t\in[0,1]} J(w_\la+t\mc M_k^\ell)-J_\la(w_\la)\right)(1+\e)<\pi.
				\end{align}Set 
				\[\zeta_0:=\lim_{k\to \infty}\left(\int_{\Om}\al(x)h(v_k)^{1-q}~dx+\la\int_\Om F(x,h(v_k))~dx\right).\] Then,
				\begin{align}\label{kk2}
					\lim_{k\to+\infty}\|v_k\|^2&=2\lim_{k\to +\infty}\left[J_\la(v_k)+\la\int_{\Om}\al(x)h(v_k)^{1-q}~dx+\la\int_\Om F(x,h(v_k))~dx\right]\notag\\
					&=2(\ga_0+\zeta_0).
				\end{align} Taking into account \eqref{kk1} and \eqref{kk2}, we deduce
				\begin{align*}
					(1+\e)\|v_k\|^2&<\frac{2\pi (\ga_0+\zeta_0)}{\ga_0-J_\la(v_\la)}=\frac{2\pi (\ga_0+\zeta_0)}{\ga_0+\zeta_0-\frac 12\|v_\la\|^2}=2\pi\left(1-\frac 12 \left(\frac{\|v_\la\|^2}{\ga_0+\zeta_0}\right)\right)^{-1}.
				\end{align*}
				Hence, we choose $r>0$ such that $\frac {(1+\e)}{2\pi}\|v_k\|^2< p:= \frac{ r}{2\pi}<\left(1- {\|u_\la\|^2}\right)^{-1},$ where $u_\la:=\frac{v_\la}{(2(\ga_0+\zeta_0))^{1/2}}$. Also, note that $u_k:=\frac{v_k}{\|v_k\|}\rightharpoonup\frac{v_\la}{(2(\ga_0+\zeta_0))^{1/2}}=u_\la$ weakly in $H^1_0(\Om)$ as $k\to+\infty.$ Therefore, using Theorem \ref{hi}, for any $\e>0$, we obtain \[\ds\sup_k\int_\Om \exp\left(4\pi p u_k^2\right)~dx<+\infty,\] which together with the fact that $p>\frac{(1+\e)}{2\pi}\|v_k\|^2$ yields 
				\begin{align}\label{he1}
					\sup_k\int_\Om \exp\left(2(1+\e)v_k^2\right) ~dx<+\infty.
				\end{align} Now using  Lemma \ref{L1}-$(h_6)$, $(f2)$ and \eqref{he1}, we get
				\begin{align}
					&\int_{\Om} f(x,h(v_k))h'(v_k)v_k~dx\notag\\&=\int_{\Om\cap\{x\,:\, v_k(x)>N>1\}} f(x,h(v_k))h'(v_k)v_k~dx+\int_{\Om\cap\{x\,:\, v_k(x)\leq N\}} f(x,h(v_k))h'(v_k)v_k~dx\notag
					\\&= O\left( \int_{\Om\cap\{x\,:\, v_k(x)>N\}} \exp \left(\left(1+\frac\e 2\right) h(v_k)^4\right) dx\right)+\int_{\Om\cap\{x\,:\, v_k(x)\leq N\}} f(x,h(v_k))h'(v_k)v_k~dx\notag\\
					&= O\left(\int_{\Om\cap \{x\,:\, v_k(x)>N\}} \exp \left(2\left(1+\frac\e2\right) v_k^2\right)~dx\right)+\int_{\Om\cap\{x\,:\, v_k(x)\leq N\}} f(x,h(v_k))h'(v_k)v_k~dx\notag\\
					&=O\left(\exp\left(-\e N^2\right) \int_{\Om\cap \{x\,:\, v_k(x)>N\}} \exp \left(2(1+\e) v_k^2\right) ~dx\right)+\int_{\Om\cap\{x\,:\, v_k(x)\leq N\}} f(x,h(v_k))h'(v_k)v_k~dx\notag\\
					& =O\left(\exp(-\e N^2)\right) +\int_{\Om\cap\{x\,:\, v_k(x)\leq N\}} f(x,h(v_k))h'(v_k)v_k~dx\notag\\&\to\int_\Om f(x,h(v_\la))h'(v_\la)v_\la ~dx, \text{ \; as } 
					k\to+\infty \text{ and } N\to+\infty,\nonumber
				\end{align} thanks to the Lebesgue dominated convergence theorem.
			Thus, we obtain \eqref{compp}. Hence,
			taking into account \eqref{mpsi} and \eqref{compp} and arguing similarly as in the proof of Proposition \ref{prp1}, we can infer that $v_k\to v_\la$ strongly in $H^1_0(\Om)$ as $k\to+\infty$ and $v_\la\not=w_\la$.
			Hence, the claim is verified. Thus, the proof of the proposition follows.
		\end{proof} 
		\noi {\bf Proof of Theorem \ref{thm2}:}  
		From Proposition \ref{prp2}, it follows that $v_\la\in H_0^1(\Om)\cap C^+_{\varphi_q}(\Omega)$ is a weak solution to \eqref{pp}. Moreover, $v_\la\not=w_\la$, where $w_\la$ is another weak solution to \eqref{pp}. Now by Lemma \ref{L1}-$(h_1)$, we have $h$ is a $C^\infty$  function and Lemma \ref{L1}-$(h_8), (h_{11})$ ensure that $h(s)$ behaves like $s$ when  $s$ is close to $0$. Therefore,  $h(v_\la) \in H_0^1(\Om)\cap C^+_{\varphi_q}(\Omega)$ forms a weak solution to the problem \eqref{pq} and $h(v_\la)\not=h(w_\la)$, where $h(w_\la)$ is another solution to the problem \eqref{pq} obtained in Theorem \ref{thm1}. This concludes the proof of Theorem \ref{thm2}. \\
		
		\section*{Acknowledgement}
		The authors are  thankful to the editor and the reviewers for their valuable suggestions.\\\\
		{\bf Funding information:} Reshmi Biswas is funded by Fondecyt Postdocotal  Project No. 3230657,   Sarika Goyal  is funded  by SERB under the grant SPG/2022/002068, K. Sreenadh is funded by  Janaki \& K. A. Iyer Chair Professor grant.\\\\
\noi{\bf Conflict of interest:} The authors state no conflict of  interest .\\\\
		\noi {\bf Data availability statement:} No data were used for the research described in the article.      
		
	\end{document}